\documentclass[11pt,leqno,latexsym,amsbsy,xypic,verbatim]{amsart}
\usepackage{amsfonts,amssymb,mathrsfs,}
\usepackage{amsmath,amscd,pstricks}
\renewcommand{\subsection}[1]{\vspace{.18in}\par\noindent\addtocounter{subsection}{1}\setcounter{equation}{0}{\bf\thesubsection.\hspace{5pt}#1}}

\theoremstyle{definition}
\newtheorem{Def}[subsection]{Definition}
\newtheorem{Bsp}[subsection]{Example}
\newtheorem{Rem}[subsection]{Remark}
\newtheorem{Rems}[subsection]{Remarks}
\theoremstyle{plain}
\newtheorem{Prop}[subsection]{Proposition}
\newtheorem{Thm}[subsection]{Theorem}
\newtheorem{Lem}[subsection]{Lemma}
\newtheorem{Coro}[subsection]{Corollary}
 \numberwithin{equation}{subsection}
\newcommand{\sg}{\sigma}
\newcommand{\tu}{{\tilde u}}
\newcommand{\bin}{\bigcup}
\newcommand{\bla}{\boldsymbol{\lambda}}
\newcommand{\be}{\boldsymbol{e}}
\newcommand{\il}{{i_0}}
\newcommand{\ile}{{i_0+1}}
\newcommand{\lk}{\left[}
\newcommand{\rk}{\right]}
\newcommand{\kong}{\varnothing}
\newcommand{\jiao}{\bigcap}
\newcommand{\han}{\subseteq}
\def\leq{\leqslant}\def\geq{\geqslant}
\newcommand{\hl}{\widehat L}
\newcommand{\soc}{\mathrm{soc}}
\newcommand{\rhol}{\rho_{r+l',r}}

\newcommand{\diag}{\mathrm{diag}}
\newcommand{\dt}{\delta}
\newcommand{\Dt}{\Delta}

\newcommand{\map}{\mapsto}

\newcommand{\vs}{\varepsilon}

\newcommand{\og}{\omega}

\newcommand{\up}{\upsilon}
\newcommand{\ep}{\varepsilon}
\newcommand{\al}{\alpha}
\newcommand{\bt}{\beta}
\newcommand{\h}{\widehat}
\newcommand{\ti}{\widetilde}

\newcommand{\zr}{\zeta_r}
\newcommand{\s}{\sigma}

\newcommand{\ot}{\otimes}
\newcommand{\ttk}{\mathtt{k}}
\newcommand{\ttF}{\mathtt{F}}

\newcommand{\tte}{\mathtt{e}}
\newcommand{\ttf}{\mathtt{f}}

\newcommand{\fS}{\frak S}
\newcommand{\fL}{\frak L}
\newcommand{\fM}{\frak M}
\newcommand{\fI}{\frak I}
\newcommand{\sZ}{\mathcal Z}
\newcommand{\sH}{\mathcal H}

\newcommand{\bm}{\bigm}
\newcommand{\ba}{\bar}
\newcommand{\ol}{\overline}
\newcommand{\lra}{\longrightarrow}
\newcommand{\ra}{\rightarrow}
\newcommand{\la}{\lambda}
\newcommand{\La}{\Lambda}
\newcommand{\lrb}{\overline{\Lambda(n,r)}_{l'}}
\newcommand{\lb}{\overline{\lambda}}
\newcommand{\mb}{\overline{\mu}}
\newcommand{\nb}{\overline{\nu}}

\newcommand{\mbn}{\mathbb N}
\newcommand{\mbq}{\mathbb Q}
\newcommand{\mbz}{\mathbb Z}
\newcommand{\bfd}{\mathbf{d}}

\newcommand{\bfz}{\mathbf{z}}

\newcommand{\bft}{\mathbf{t}}

\newcommand{\bfU}{\mathbf{U}}

\newcommand{\ttp}{\mathtt{p}}

\newcommand{\Ext}{\operatorname{Ext}}

\newcommand{\Hom}{\operatorname{Hom}}

\newcommand{\sB}{\mathcal{B}}
\def\sfK{{\mathsf K}}
\def\sfX{{\mathsf X}}
\newcommand{\hlr}{\widehat L_h}
\begin{document}
\title{Representations of little $q$-Schur algebras}
\author{Jie Du, Qiang Fu and Jian-pan Wang}
\address{School of Mathematics, University of New South Wales,
Sydney 2052, Australia.} \email{j.du@unsw.edu.au\ \,\, {\it Home
page:} \tt http://web.maths.unsw.edu.au/$\sim$jied}
\address{Department of Mathematics, Tongji University, Shanghai, 200092, China.}
\email{q.fu@hotmail.com}
\address{Department of Mathematics, East China Normal University, Shanghai, 200062, China.}
\email{jpwang@ecnu.edu.cn}
\thanks{Supported by the Natural Science Foundation of China (Grant No. 10971154), the Program NCET and the UNSW FRG}
\date{\today}
\sloppy \maketitle
\begin{abstract}
In \cite{DFW} and \cite{Fu07}, little $q$-Schur algebras were introduced as
homomorphic images of the infinitesimal quantum groups. In this paper,
we will investigate representations of these algebras. We will classify simple  modules for little
$q$-Schur algebras and classify semisimple little $q$-Schur algebras.
Moreover, through the classification of the blocks of little $q$-Schur algebras
for $n=2$, we will determine little $q$-Schur algebras of finite representation type in the odd roots of unity case.
\end{abstract}

\section{Introduction}

The $q$-Schur algebras are certain finite dimensional algebras which were used by Jimbo in the establishment of the quantum Schur--Weyl reciprocity \cite[Prop.~3]{JI} and were introduced by Dipper and James in
\cite{DJ1},\cite{DJ2} in the study the representations of Hecke algebras and finite general linear groups. Using a geometric setting for
$q$-Schur algebras, Belinson, Lusztig and  MacPherson \cite{BLM}
reconstructed (or realized) the quantum enveloping algebra $\bfU(n)$ of $\frak{gl}_n$ as a limit of $q$-Schur algebras over $\mathbb Q(\upsilon)$. This results in an explicit description of the
epimorphism $\zr$ from $\bfU(n)$ to the $q$-Schur algebra $\bfU(n,r)$ for all $r\geq 0$.
Restriction induces an epimorphism from the Lusztig form $U_\sZ(n)$ over $\sZ=\mbz[\up,\up^{-1}]$ to the integral $q$-Schur algebra $U_\sZ(n,r)$ \cite{Du} and, in particular, an epimorphism from $U_k(n)$ to $U_k(n,r)$
by specializing the parameter to any root of unity in a field $k$. The little $q$-Schur algebras $\tu_k(n,r)$ are defined as the homomorphic images of the finite dimensional Hopf subalgebra $\tu_k(n)$ of $U_k(n)$ under $\zeta_r$.
The structure of these algebras are investigated by the authors (\cite{DFW},\cite{Fu07}). For example,
through a BLM type realization for $\tu_k(n)$, various bases for $\tu_k(n,r)$ were constructed and dimension formulas were given. This paper is a continuation of \cite{DFW} and \cite{Fu07}.

It should be point out that, through the coordinate algebra approach,  Doty, Nakano and Peters in \cite{DNP1} defined infinitesimal
Schur algebras, closely related to the Frobenius kernels of an
algebraic group over a field of positive characteristic. A theory
for the quantum version of the infinitesimal Schur algebras was
studied by Cox in \cite{Cox} and \cite{Cox00}.  The
relation between the algebra structures of the infinitesimal and
the little $q$-Schur algebra was investigated in \cite{Fu05}.
Indeed, we see that the relation between the little and the
infinitesimal $q$-Schur algebra is similar to that between the
$h$-th Frobenius kernel $G_h$ and the corresponding Jantzen
subgroup $G_hT$. However, there is a subtle difference between infinitesimal $q$-Schur algebras and little $q$-Schur a1gebras.

Suppose $\ep$ is an $l'$-th root of 1 in a field $k$ and define $l=l'$ if $l'$ is
odd, and $l=l'/2$, if $l'$ is even.
 The parameter involved for defining $G_h$, $G_hT$ and the infinitesimal $q$-Schur algebras is $q=\ep^2$ which is always an $l$-th root of 1. So their representations are independent of $l'$ (cf. \cite[3.1]{Don} or \ref{the relation between various irreducible module} below). However, the parameter used for defining
little $q$-Schur algebras is $\ep$ which is a square root of $q$.\footnote{For this reason, little $q$-Schur algebras should probably be more accurately renamed as little $\sqrt{q}$-Schur algebras.} The structures and representations of $\tu_k(n)$ and little $q$-Schur algebras {\it do } depend on $l'$ and are quite different (cf. \cite[5.11]{L2} or \ref{the irreducible module of u(n)} below). In fact, the interesting case is the even case, where simple representations of $\tu_k(n)$ are indexed by $(\mathbb Z_{l'})^n$ and not every simple module can be obtained as a restriction of a simple module of $U_k(n)$ with an $l$-restricted highest weight. In contrast with the algebraic group case, this is a kind of ``quantum phenomenon''. Moreover, it should be noted that the representation theory of quantum enveloping algebras at the even roots of unity
has found some new applications in the conformal field theory (or the theory of vertex operator algebras).
See, e.g., \cite{GSTF1} and \cite{KS}.

We will show that every simple module of $\tu_k(n)$ is a restriction of a simple  $G_1T$-module.
To achieve this, we first classify simple $\tu_k(n,r)$-modules through the ``sandwich'' relation
$u_k(n,r)_1\subseteq\tu_k(n,r)\subseteq s_k(n,r)_1$ given in \eqref{sandwich}. By introducing the baby transfer map (cf. \cite{L00}), we will see a simple $\tu_k(n,r)$-module for $n\geq r$ and $l'$ odd is either an inflation of a simple
$\tu_k(n,r-l')$ via the module transfer map or a lifting of a simple module of the Hecke algebra via the $q$-Schur functor.
Main results of the paper also include the classifications of semisimple little $q$-Schur algebras and, when $l'$ is odd, the finite representation type of little $q$-Schur algebras.

We organize the paper as follows. We recall the definition for the
infinitesimal quantum groups $\tu_k(\frak g)$ associated with a simple Lie algebra $\frak g$ of a simply-laced type and $\tu_k(n)$ associated with $\frak {gl}_n$ in \S2. In \S3, we shall study the
baby Weyl module for infinitesimal quantum group $\tu_k(\frak g)$.  We shall prove in \ref{the little
weyl module of u'(g)} that for the restricted weight, the
corresponding baby Weyl module is equal to the Weyl module. This is a well-known fact.
Furthermore, using this result we can give another proof of
\cite[7.1(c)(d)]{L1}.  In \S4, we shall recall
some results about the little and the infinitesimal $q$-Schur
algebra and establish the sandwich relation mentioned above. Moreover, classifications of simple $G_h$- and $G_hT$-modules from  \cite{Don}, \cite{Cox} and \cite{Cox00} will be mentioned.
 In \S5, we shall study the
baby Weyl module for the infinitesimal quantum group $\tu_k(n)$ and give the classification of simple  module for
the little $q$-Schur algebra $\tu_k(n,r)$. The baby transfer maps are discussed in \S6. In \S7,
we shall classify semisimple little
$q$-Schur algebras, while in \S8 we classify the finite representation type of little $q$-Schur algebras
 through the classification of the blocks of little
$q$-Schur algebras for $n=2$. Finally, in an Appendix, we show that
the epimorphism from $U_\sZ(n)$ onto $U_\sZ(n,r)$ remains surjective when restricted to the
 Lusztig form  $U_\sZ(\frak {sl}_n)$ of the quantum $\mathfrak{sl}_n$. Thus,
 the results developed in \S3 for $\mathfrak{sl}_n$ can be directly used in \S5.

Throughout, let $\upsilon$ be an indeterminate and let
$\sZ=\mathbb Z[\upsilon,\upsilon^{-1}]$. Let $k$ be a field
containing a primitive $l'$th root $\ep$ of $1$ with $l'\geq 3$.
Let $l>1$ be defined by
$$l=
\begin{cases}
l'&\text{if $l'$ is
odd},\\
l'/2&\text{if $l'$ is even}.
\end{cases}$$
Thus, $\ep^2$ is always a  primitive $l$th root of 1.
Specializing $\up$ to $\ep$, $k$ will be viewed as a
$\sZ$-module.

For a  finite dimensional algebra $A$ over $k$, let $\text{Mod}(A)$ be the category of finite dimensional left $A$-modules. If $B$ is a quotient algebra of $A$, then the {\it inflation functor} embeds $\text{Mod}(B)$ into $\text{Mod}(A)$ as a full subcategory.

\section{Lusztig's infinitesimal quantum enveloping algebras}
In this section, following \cite{L2}, we recall the definition for
the infinitesimal quantum group.

Let $\frak g$ be a semisimple complex Lie algebra associated with an
indecomposable positive definite symmetric Cartan matrix
$C=(a_{ij})_{1\leq i,j\leq n}$.

\begin{Def} \label{the definition of U(g)}
The quantum enveloping algebra of $\frak g$ is the algebra
$\bfU(\frak g)$ over $\mathbb Q(\upsilon)$ generated by the
elements
$$E_i,\ F_i,\ {\sfK}_i,\ \sfK_i^{-1}\ (1 \leq i \leq n)$$
subject to the following relations $:$
\begin{itemize}
\item[(QG1)]
$\ \sfK_{i}\sfK_{j}=\sfK_{j}\sfK_{i},\ \sfK_{i}\sfK_{i}^{-1}=1;$
\item[(QG2)]
$\ \sfK_iE_j=\upsilon^{a_{ij}}E_j\sfK_i,\ \sfK_iF_j=\upsilon^{-a_{ij}}
F_j\sfK_i;$
\item[(QG3)]
$\ E_iF_j-F_jE_i=\delta_{ij}\frac
{\sfK_{i}-\sfK_{i}^{-1}}{\upsilon-\upsilon^{-1}};$
\item[(QG4)]
$\ E_iE_j=E_jE_i,\ F_iF_j=F_jF_i,$ if $a_{ij}=0;$
\item[(QG5)]
$\ E_i^2E_j-(\upsilon+\upsilon^{-1})E_iE_jE_i+E_jE_i^2=0,$ if $a_{ij}=-1;$
\item[(QG6)]
$\ F_i^2F_j-(\upsilon+\upsilon^{-1})F_iF_jF_i+F_jF_i^2=0,$ if $a_{ij}=-1.$
\end{itemize}
\end{Def}

Definition \ref{the
definition of U(g)} implies immediately the following result.
\begin{Lem}\label{the algebra automorphism sigma on Uupsilon(g)}
There is a unique $\mbq(\up)$-algebra automorphism $\s$ on
$\bfU(\frak g)$ satisfying
$$\s(E_i)=F_i,\ \s(F_i)=E_i,\ \s(\sfK_i)=\sfK_i^{-1}.$$
\end{Lem}

For any integers
$c,t$ with $t\geq 1$, let $[c]=\frac{\up^c-\up^{-c}}{\up-\up^{-1}} \in\sZ,$ $[t]^{!}=[1][2]\cdots[t],$ and
$$\bigg [ { c \atop t} \bigg ]=\prod _{s=1} ^{t}\frac
{\upsilon^{c-s+1}-\upsilon^{-c+s-1}}
{\upsilon^{s}-\upsilon^{-s}}\in\sZ.$$
If we put $[0]^{!}=0=\left [ { c \atop 0} \right ]$, then $\left[ { c \atop t} \right]=[c]!/[t]![c-t]!$ for  $c\geq t\geq 0$, and $\left[ { c \atop t} \right]=0$ for $t>c\geq
0$.

Let $ U_\sZ(\frak g)$
(resp., $ U_\sZ^+(\frak g), U_{\sZ}^-(\frak g))$ be the
$\sZ$-subalgebra of $\bfU(\frak g)$ generated by the elements
$E_i^{(N)}=E_i^N/[N]!,$ $F_i^{(N)}=F_i^N/[N]!$, and $\sfK_j^{\pm 1}$
$(1\leq i\leq n-1,\ 1\leq j\leq n,\ N\geq 0)$ (resp., $E_i^{(N)},\
F_i^{(N)})$. Let $\mathit U_\sZ^0(\frak g)$ be the
$\sZ$-subalgebra of $\bfU(\frak g)$ generated by all $\sfK_i^{\pm 1}$
and $\big[ {\sfK_i;0 \atop t}\big]$, where for $t\in\mathbb N$ and
$c\in\mathbb Z$,
$$\bigg[ {\sfK_i;c \atop t} \bigg] =
         \prod_{s=1}^t \frac
         {\sfK_i\upsilon^{c-s+1}-\sfK_i^{-1}\upsilon^{-c+s-1}}{\upsilon^s-\upsilon^{-s}}.$$




\begin{Prop}
The following identities hold in $U_\sZ(\frak g)$.
\begin{equation}\label{the communicative formula Fi in Uupsilon(g)}
F_i^{(N)}F_j^{(M)}=\sum_{N-M\leq s\leq
N}(-1)^{s+N-M}\left[{s-1\atop
N-M-1}\right]F_i^{(N-s)}F_j^{(M)}F_i^{(s)},
\end{equation}
\begin{equation}\label{the communicative formula Fi1 in Uupsilon(g)}
F_j^{(M)}F_i^{(N)}=\sum_{N-M\leq s\leq
N}(-1)^{s+N-M}\left[{s-1\atop
N-M-1}\right]F_i^{(s)}F_j^{(M)}F_i^{(N-s)},
\end{equation}
 where $N>M\geq 0$ and
$a_{ij}=-1$.
\end{Prop}
\begin{proof}
Applying the algebra automorphism $\s$ given in \ref{the algebra
automorphism sigma on Uupsilon(g)} to \cite[2.5(a),(b)]{L2}, we
get the desired formulas.
\end{proof}

Regarding the field $k$ as a
$\sZ$-algebra by specializing $\upsilon$ to $\ep$, we will write $[t]_\ep$ and $\lk{c\atop t}\rk_\ep$ for the images of $[t]$ and $\lk{c\atop t}\rk$ in $k$, and define, following \cite{L2}, the $k$-algebras  $U_k^+(\frak g)$,
$U_k^-(\frak g)$, $U_k^0(\frak g)$, and $U_k(\frak g)$ by applying the functor
$(\ )\otimes_\sZ k$ to $U_\sZ^+(\frak g)$, $U_\sZ^-(\frak g)$,
$U_\sZ^0(\frak g)$, and $U_\sZ(\frak g)$.   We will denote the images of
$E_i$, $F_i$, etc. in $U_k(\frak g)$ by the same letters.

Let $\tu_k^+(\frak g)$, $\tu_k^-(\frak g)$, $\tu_k^0(\frak g)$,
and $\tu_k(\frak g)$ be the $k$-subalgebras of $U_k(\frak g)$
generated respectively by the elements $E_i\ (1\leq i\leq n)$,
$F_i\ (1\leq i\leq n)$, $\sfK_i^{\pm 1}\ (1\leq i\leq n)$, and $E_i$,
$F_i$, $\sfK_i^{\pm 1}\ (1\leq i\leq n)$. By a proof similar to \cite[Th.~2.5]{DFW},
$\tu_k(\frak g)$ can be presented by generators $E_i$,
$F_i$, $\sfK_i^{\pm 1}\ (1\leq i\leq n)$ and the relations (QG1)--(QG6) together with
$$E_i^l=0=F_i^l,\quad \sfK_i^{2l}=1.$$

 The algebra $\tu_k(\frak
g)$ is called {\it the infinitesimal quantum group} associated with
$\frak{g}$.
When $l'=l$ is odd, we will also call the algebra $u_k(\frak g)=\tu_k(\frak g)/\langle
\sfK_1^{l}-1,\cdots,\sfK_n^{l}-1\rangle$, considered in \cite{L2},  an infinitesimal quantum group.

For the reductive complex Lie algebra $\frak{g}=\frak{gl}_n$, we now modify the definitions above
to introduce infinitesimal quantum $\frak{gl}_n$ which will be used to define little $q$-Schur algebras in \S4.

Let $\bfU(n)=\bfU(\frak{gl}_n)$ be the quantum enveloping
algebra of $\frak{gl}_n$ which is a slightly modified version of
Jimbo \cite{JI}; see \cite[3.2]{T}. It is generated by the
elements $E_i,F_i$ $(1\leq i\leq n-1)$ and $K_i^{\pm 1}$ $(1\leq i\leq
n)$ subject to the relations as given in \cite[Def.~2.1]{DFW}.

Let $\ti{K}_i=K_iK_{i+1}^{-1}$ for $1\leq i\leq n-1$. Then
the subalgebra $'\bfU(n)$ of $\bfU(n)$ generated by the $E_i,F_i$ and
$\ti{K}_i$ $(1\leq i\leq n-1)$ is isomorphic to the quantum enveloping algebra
$\bfU(\mathfrak{sl}_n)$. By identifying $\ti{K}_i$ with $\sfK_i$, we will identify
$'\bfU(n)$ with $\bfU(\mathfrak{sl}_n)$.

Following \cite{T}, let $U_\sZ(n)$ (resp.,
$U_\sZ^+(n),\ U_\sZ^-(n)$) be the $\sZ$-subalgebra of $\bfU(n)$
generated by all $E_i^{(m)}$, $F_i^{(m)}$, $K_i$ and $\left[
{K_i;0 \atop t} \right]$, (resp., $E_i^{(m)}$, $F_i^{(m)}$). Let
$U_\sZ^0(n)$ be the $\sZ$-subalgebra of $\bfU(n)$ generated by all
$K_i$ and $\left[ {K_i;0 \atop t} \right]$.
Replacing $K_i$ by $\ti K_i$, we may defined integral forms
$'U_\sZ(n)$, which is identified with $U_\sZ(\mathfrak{sl}_n)$, and define ${}'U_\sZ^0(n)$  and ${}'U_\sZ^\pm(n)=U_\sZ^\pm(n)$ similarly.

Let $U_k(n)=U_\sZ(n)\ot_\sZ k$ and $'U_k(n)={}'U_\sZ(n)\ot_\sZ k$. Since
$'U_\sZ(n)$ is a pure $\sZ$-submodule of $U_\sZ(n)$ (see \cite[Prop.~2.6]{Du}),
$'U_k(n)$ is a subalgebra of $U_k(n)$ identified with $U_k(\mathfrak{sl}_n)$.

Following \cite{L2}, let $\tu_k(n)$ be the $k$-subalgebra of $U_k(n)$ generated by the
elements $E_i$, $F_i$, $K_i^{\pm 1}$ ($1\leq i\leq n$).  Let
$\tu_k^+(n)$, $\tu_k^0(n)$, $\tu_k^-(n)$ be the
$k$-subalgebra of $\tu_k(n)$ generated respectively by the
elements $E_i$ $(1\leq i\leq n-1)$, $K_j^{\pm 1}$ $(1\leq j\leq n)$,
$F_i$ $(1\leq i\leq n-1)$.  We shall denote the images of $E_i$,
$F_i$, etc. in $U_k(n)$, $\tu_k(n)$  by the same letters. In
the case of $l'$ being an odd number, let
$$u_k(n)=\tu_k(n)/\langle K_1^l-1,\cdots,K_n^l-1\rangle.$$

Similarly, we can define $'\tu_k(n)$ etc. as subalgebras of $'U_k(n)$, which are identified with
$\tu_k(\mathfrak{sl}_n)$ etc.

\section{Baby Weyl modules}
Following \cite[5.15]{Jan96} for $\bfd=(d_1,\cdots,d_n)\in\mbn^n$ the
$\bfU(\frak g)$-module $V(\bfd)=\bfU(\frak g)/I(\bfd)$ is
irreducible where
$$I(\bfd)=\sum_{1\leq i\leq n}(\bfU(\frak
g)E_i+\bfU(\frak g)F_i^{d_i+1}+\bfU(\frak g)(\sfK_i-\up^{d_i})).$$
Let $x_0=1+I(\bfd)\in L(\bfd)$. Let $V_\sZ(\bfd)$ be the
$U_\sZ(\frak g)$-submodule of $V(\bfd)$ generated by $x_0$. Let
$V_k(\bfd)=V_\sZ(\bfd)\ot_\sZ k$. This is the Weyl module of
$U_k(\frak g)$ with highest weight $\bfd$. For convenience, we shall denote the image of
$x_0$ in $V_k(\bfd)$ by the same letter.
We call the $\tu_k(\frak g)$-module $V'_k(\bfd):=\tilde u_k(\frak g)x_0$ the {\it baby Weyl module} of $U_k(\frak g)$ (or the Weyl module of $\tu_k(\frak g)$).

\begin{Lem}\label{the communicative formula in u'(g)} Let $N\geq 0$ be an integer.
\begin{itemize}
\item[(1)] If $Y\in\tilde u_k(\frak g)$ is a monomial in the $F_i$'s, then
\begin{equation*}
F_i^{(N)}Y=\sum_{s\geq 0}X_sF_i^{(s)},\quad\text{ for some }X_s\in \tilde u_k^-(\frak g).
\end{equation*}

\item[(2)] If $Y\in \tilde u_k(\frak g)$ is a monomial in the $E_i$'s, then
have
\begin{equation*}
YE_i^{(N)}=\sum_{s\geq 0}E_i^{(s)}X_s,\quad\text{ for some }X_s\in \tilde u_k^+(\frak g).
\end{equation*}
\end{itemize}
\end{Lem}

\begin{proof} We only prove (1). The proof of (2) is similar.

Assume $Y=F_{j_1}^{(M_1)}F_{j_2}^{(M_2)}\cdots
F_{j_t}^{(M_t)}$ where $0\leq M_i<l$ for all $i$. We proceed by
induction on $t$.

Suppose $t=1$ and $M<l$. If $N\leq M$, then $F_i^{(N)}F_j^{(M)}\in
\tilde u_k^-(\frak g)$, where $j=j_1$. Hence the result follows by putting
$X_0=F_i^{(N)}F_j^{(M)}$. We now assume $0\leq M<N$. If
$a_{ij}=0$, then $F_i^{(N)}F_j^{(M)}=F_j^{(M)}F_i^{(N)}$ by the
definition \ref{the definition of U(g)}. If $a_{ij}=-1$, then, by
\eqref{the communicative formula Fi in Uupsilon(g)},
\begin{equation*}
F_i^{(N)}F_j^{(M)}=\sum_{N-M\leq s\leq
N}(-1)^{s+N-M}\left[{s-1\atop
N-M-1}\right]F_i^{(N-s)}F_j^{(M)}F_i^{(s)}=\sum_{N-M\leq s\leq N}X_sF_i^{(s)},
\end{equation*}
where $X_s=(-1)^{s+N-M}\left[{s-1\atop
N-M-1}\right]F_i^{(N-s)}F_j^{(M)}$ for $N-M\leq s\leq N$. Note
that if $N-M\leq s$ with $M< l$ then $N-s< l$.  Hence, $X_s\in \tilde u_k^-(\frak g)$.

Assume now $t>1$. Let $Y'=F_{j_1}^{(M_1)}F_{j_2}^{(M_2)}\cdots
F_{j_{t-1}}^{(M_{t-1})}$. Then by induction, we have
$$F_i^{(N)}Y'=\sum_{s\geq 0}X_sF_i^{(s)},$$
where $X_s\in \tilde u_k^-(\frak g)$. The previous argument show
that for each $s\geq 0$, we have
$$F_i^{(s)}F_{j_t}^{(M_t)}=\sum_{m\geq 0}Y_{s,m}F_i^{(m)},\text{ for some }Y_{s,m}\in \tilde u_k^-(\frak g).$$
 Let $X_m'=\sum_{s\geq
0}X_sY_{s,m}\in \tilde u_k^-(\frak g)$ for $m\geq 0$. Then
$$F_i^{(N)}Y=\sum_{m\geq 0}X_m'F_i^{(m)},$$
as required.
\end{proof}
\begin{Thm}\label{the little weyl module of u'(g)}
For $\bfd=(d_1,\cdots,d_n)\in \mathbb N^n$ with $0\leq d_i<l$ for all $i$,
we have $V'_k(\bfd)=V_k(\bfd)$.
\end{Thm}
\begin{proof}By the definition of $V_k(\bfd)$ and $V'_k(\bfd)$,
we have $V_k(\bfd)=U_k^-(\frak g)x_0$ and $V'_k(\bfd)=\tilde
u_k^-(\frak g)x_0$. Since $U_k^-(\frak g)$ is generated by the
elements $F_i^{(N)}$ ($1\leq i\leq n,\ N\geq 0$), it is enough to
prove $F_i^{(N)}\tilde u_k^-(\frak g)x_0\han \tilde u_k^-(\frak
g)x_0$ for all $1\leq i\leq n$ with $N\geq 0$. For a monomial
$Y$ of $F_i$ in $\tilde u_k(\frak g)$, by \ref{the communicative
formula in u'(g)}(1), we have
\begin{equation}\label{the communicative formula in Vk(g)}
F_i^{(N)}Yx_0=\sum_{s\geq 0}X_sF_i^{(s)}x_0,
\end{equation}
where $X_s\in \tilde u_k^-(\frak g)$ and $N\geq 0$. Since $0\leq
d_i<l$ and $F_i^{(d_i+1)}x_0=0$ for all $i$, we have
$F_i^{(s)}x_0=0$ for $s\geq l$. By \eqref{the communicative
formula in Vk(g)}, we have
$$F_i^{(N)}Yx_0=\sum_{0\leq s<l}X_sF_i^{(s)}x_0\han \tilde u_k^-(\frak g)x_0.$$
Hence the result follows.
\end{proof}
Following \cite[4.6]{L1}, we say that a $U_k(\frak g)$-module  $V$
has type $\bf 1$ if $V=\{v\in V\mid \sfK_i^lv=v\ \text{for}\
i=1,\cdots,n\}$. Let $V$ be a $U_k(\frak g)$-module of type $\bf
1$. For any $\bfz=(z_1,\cdots,z_n)\in\mbz^n$, following
\cite[5.2]{L1}, we define the $\bfz$-weight space
$$V_{\bfz}=\left\{x\in V\mid \sfK_ix=\ep^{z_i}x,\
\lk{\sfK_i;0\atop l}\rk x=\lk{z_i\atop l}\rk_\ep x\ for \
i=1,\cdots,n\right\}.$$

\begin{Lem}\label{Ki;c l-Ki;c' l}
$($\cite[4.2]{L1}$)$Let $V$ be a $U_k(\frak g)$-module
and let $x\in V$ be such that $\sfK_ix=\ep^mx$ for some $m\in\mbz$.
Then for any $c,c'\in\mbz$ we have $\lk{\sfK_i;c\atop l}\rk
x-\lk{\sfK_i;c'\atop l}\rk x=\left(\lk{m+c\atop
l}\rk_\ep-\big[{m+c'\atop l}\big]_\ep\right)x\in\mbz x$.
\end{Lem}
Using \ref{Ki;c l-Ki;c' l}, we see that $\lk{\sfK_i;c\atop l}\rk
x=\lk{z_i+c\atop l}\rk_\ep x$, $x\in V_\bfz,\ c\in\mbz$. Let
$\al(i)=(a_{1i},a_{2i},\cdots,a_{ni})\in\mbz^n$, $(1\leq i\leq
n)$. We define a partial order on $\mbz^n$ by $\bfz\leq
\bfz'\Leftrightarrow \bfz'-\bfz=\sum\limits_{i=1}^nc_i\al(i)$ for
some $c_1,\cdots,c_n\in\mbn$. This is a partial order on $\mbz^n$.
It is clear that we have the following lemma.
\begin{Lem}\label{the action of Ei and Fi on Vz}
Let $V$ be a $U_k(\frak g)$-module of type $\bf 1$.
Then for $N>0$, we have

$(1)$ $E_i^{(N)}V_\bfz\han V_{\bfz+N\al(i)}$

$(2)$ $F_i^{(N)}V_\bfz\han V_{\bfz-N\al(i)}$.
\end{Lem}

By \cite[6.2]{L1}, for $\bfd=(d_1,\cdots,d_n)\in\mbn^n$,
$V_k(\bfd)$ has a unique maximal $U_k(\frak g)$-submodule
$W_k(\bfd)$. Let $L_k(\bfd)=V_k(\bfd)/W_k(\bfd)$. Then $L_k(\bfd)$
is a simple $U_k(\frak g)$-module. Similarly, the $\tilde u_k(\frak g)$-module
$V_k'(\bfd)$ has a unique maximal $\tilde u_k(\frak g)$ submodule
$W_k'(\bfd)$ by the proof of \cite[5.11]{L2}. Let $L_k'(\bfd)=V_k'(\bfd)/W_k'(\bfd)$. Then
$L_k'(\bfd)$ is a simple $\tilde u_k(\frak g)$-module.

Let $I^-$ be the ideal in $\tilde u_k^-(\frak g)$ spanned as a
$k$-vector space by the nonempty words in $F_i$, $1\leq i\leq n$.
\begin{Lem}\label{I^-x0}
Assume $\bfd\in\mbn^n$ with $d_i<l$ for all $i$. Let $x_0$ be the
highest weight vector of $V(\bfd)$. Then we have
$$I^-x_0=\sum_{\bfz<\bfd}V_k(\bfd)_\bfz.$$
\end{Lem}
\begin{proof}It is clear that $I^-x_0\han
\sum_{\bfz<\bfd}V_k(\bfd)_\bfz$. Since $V_k(\bfd)=kx_0\oplus
\sum_{\bfz<\bfd}V_k(\bfd)_\bfz$ and $V_k'(\bfd)=kx_0\oplus
I^-x_0$, by \ref{the little weyl module of u'(g)}, we have
$kx_0\oplus \sum_{\bfz<\bfd}V_k(\bfd)_\bfz=kx_0\oplus I^-x_0$.
Hence dim\,$\sum_{\bfz<\bfd}V_k(\bfd)_\bfz=$dim\,$I^-x_0$, and the
result follows.
\end{proof}
\begin{Lem}\label{I^-x0 is stabilized undet the left action by elements EiN with N geq l}
Assume $\bfd\in\mbn^n$ with $d_i<l$ for all $i$. Let $x_0$ be the
highest weight vector of $V(\bfd)$. Then, $E_i^{(N)}I^-x_0\han I^-x_0$
whenever $N\geq l$.
\end{Lem}
\begin{proof}Let $Y=F_{j_1}^{M_1}F_{j_2}^{M_2}\cdots F_{j_s}^{M_s}$ where
$0<M_j<l$ for all $j$. If $E_i^{(N)}Yx_0\not\in I^-x_0$, then by
\ref{I^-x0} we have $E_i^{(N)}Yx_0\in V_k(\bfd)_\bfd$. By \ref{the
action of Ei and Fi on Vz}, we have
$N\al(i)=M_1\al(j_1)+\cdots+M_s\al(j_s)$. Since
$\al(1),\cdots,\al(n)$ are linearly independent, we have
$N=M_1+\cdots+M_s$ and $j_1=\cdots=j_s=i$. So $Y=F_i^N$. Since
$F_i^l=[l]_\ep!F_i^{(l)}=0$ and $N\geq l$, we have $Y=0$. This is
a contradiction.
\end{proof}

The following result is given in \cite[7.1(c)(d)]{L1} when $l'$ is odd.

\begin{Thm}\label{Lk=Lk'}
Assume that $\bfd=(d_1,\cdots,d_n)\in\mbn^n$ with
$0\leq d_i< l$ for all $i$. Then $L_k(\bfd)=L_k'(\bfd)$.
\end{Thm}
\begin{proof}By \ref{the little weyl module of u'(g)}, it is
enough to prove that $W_k(\bfd)=W_k'(\bfd)$. Also, by \ref{the little
weyl module of u'(g)}, the restriction of $W_k(\bfd)$ to $\tilde
u_k(\frak g)$ is a submodule of $V_k(\bfd)=V_k'(\bfd)$. Hence, by
the maximality of $W_k'(\bfd)$, we have $W_k(\bfd)\han
W_k'(\bfd)$. On the other hind, we consider the $U_k(\frak
g)$-submodule $V$ of $V_k(\bfd)$ generated by $W_k'(\bfd)$. We
shall prove that $V\han \sum_{\bfz<\bfd}V_k(\bfd)_\bfz$. Since
$W_k'(\bfd)$ is a $\tilde u_k(\frak g)$-module, by \ref{the
communicative formula in u'(g)}(2), we have
\begin{equation*}
\begin{split}
U_k^+(\frak g)W_k'(\bfd) &\han \text{span}\{E_{i_1}^{(N_1)} \cdots
E_{i_s}^{(N_s)}W_k'(\bfd)\mid s\geq 0,\ N_i\geq l\ for\ all\ i\}\\
&\han \{E_{i_1}^{(N_1)} \cdots E_{i_s}^{(N_s)}I^-x_0\mid s\geq 0,\
N_i\geq l\ for\ all\ i\}\\
&\han I^-x_0\qquad\qquad (\text{by \ref{I^-x0 is stabilized undet
the left action by elements EiN with N geq l}}).
\end{split}
\end{equation*}
Hence,  by \ref{I^-x0} and \ref{the action of Ei and Fi on Vz},
$V =U_k^-(\frak g)U_k^+(\frak g)W_k'(\bfd)\han U_k^-(\frak
g)I^-x_0=U_k^-(\frak
g)\sum_{\bfz<\bfd}V_k(\bfd)_\bfz=\sum_{\bfz<\bfd}V_k(\bfd)_\bfz$.
So by the maximality of $W_k(\bfd)$, we have $W_k'(\bfd)\han V\han
W_k(\bfd)$. The result follows.
\end{proof}
\begin{Rem}
Note that, if $l'$ is odd and $\bfd\in\mbn^n$ with $d_i<l$ for all $i$,
then $L_k'(\bfd)=L_k(\bfd)$ is also a $u_k(\frak g)$-module. So, by
\cite[6.6]{L2}, the $u_k(\frak g)$-module $L_k(\bfd)$
($\bfd\in\mbn^n,\ 0\leq d_i<l$ for all $i$) give all simple
$u_k(\frak g)$-modules.
\end{Rem}

\section{The infinitesimal and little  $q$-Schur algebras}
In this section, we shall recall the definitions of the infinitesimal
$q$-Schur algebra defined in \cite{Cox} and \cite{Cox00} and the little
$q$-Schur algebra defined in \cite{DFW,Fu07}.

For the moment, we assume that $R$ is a ring and $q^{\frac12}\in R.$

Following \cite{DD}, let $A_{q}(n)$
be the $R$-algebra generated by the $n^2$ indeterminates $c_{ij}$,
with $1\leq i,j\leq n$, subject to the relations
\begin{align*}
c_{ij}c_{il}&= c_{il}c_{ij}& &\text{for all }i,j,l,\\
c_{ij}c_{rs}&=q c_{rs}c_{ij}& &\text{for $i>r$ and
$j\leq s$},\\
c_{ij}c_{rs}&=(q-1)c_{rj}c_{is}+c_{rs}c_{ij}& &\text{for
$i>r$ and $j>s$}.
\end{align*}
The algebra $A_{q}(n)$ has a bialgebra structure such that
the coalgebra structure is given by
$$\Dt(c_{ij})=\sum_{t=1}^nc_{it}\ot c_{tj}\qquad\text{and}\qquad
\epsilon(c_{ij})=\dt_{ij}.$$
Let $A_q(n,r)$ denote the subspace of elements in $A_q(n)$ of
degree $r$. Then $A_q(n,r)$ are in fact subcoalgebras of $A_q(n)$
for all $r$, and hence, $U_R(n,r):=A_q(n,r)^*$ is an $R$-algebra, which is call a $q$-{\it Schur algebra}.

Let $\Xi(n)$ be the set of all $n\times n$ matrices over
$\mathbb N$.
Let $\sigma:\Xi(n)\rightarrow\mathbb N$ be the map sending a matrix to
the sum of its entries. Then, for $r\in\mathbb N$, the inverse
image $\Xi(n,r):=\sigma^{-1}(r)$ is the set of $n\times n$
matrices in $\Xi(n)$ whose entries sum to $r$.

For $A\in\Xi(n)$, let
$$c^A=c_{1,1}^{a_{1,1}}c_{2,1}^{a_{2,1}}\cdots
c_{n,1}^{a_{n,1}}c_{1,2}^{a_{1,2}}c_{2,2}^{a_{2,2}}\cdots
c_{n,2}^{a_{n,2}}\cdots
c_{1,n}^{a_{1,n}}c_{2,n}^{a_{2,n}}\cdots
c_{n,n}^{a_{n,n}}\in A_q(n).
$$
Then by \cite{DD} (see also \cite{T1}), the set
$\{c^A\mid A\in\Xi(n,r)\}$
forms an $R$-basis for $A_q(n,r)$.
Putting ${\xi_A}:= (c^A)^*$, we obtain the dual basis $\{\xi_A\mid A\in\Xi(n,r)\}$ for the $q$-Schur algebra $U_R(n,r)$.

For $A\in\Xi(n,r)$,  let
$$[A]=q^{\frac{-d_A}2}\xi_A\quad\text{ where }\quad d_A=-\sum_{i<s,\,j>t}a_{i,j}a_{s,t}+\sum_{j>t}a_{i,j}a_{i,t}.$$ Then $\{[A]\}_{A\in\Xi(n,r)}$ forms also a basis for $U_R(n,r)$.

We now introduce the infinitesimal $q$-Schur algebras. Thus, we assume $R=k$ is a field of {\it characteristic $p>0$} and $q=\vs^2\in k$. Since $\ep$ is a primitive $l'$-th root of unity, $q$ is {\it always} a primitive $l$-th root of unity.

Consider the following ideals in $A_q(n)$
$$\aligned
I_h&=\langle c_{ij}^{lp^{h-1}},c_{ij}^{lp^{h-1}}-1\mid 1\leq i\neq j\leq n\rangle,\\
\tilde I_h&=\langle c_{ij}^{lp^{h-1}},c_{ii}^{l'p^{h-1}}-1\mid1\leq i\not=j\leq n\rangle,\text{ and }\\
J_h&=\langle c_{ij}^{lp^{h-1}}\mid1\leq i\neq j\leq n\rangle.\\
\endaligned$$
Clearly, $J_h\han \tilde I_h\han I_h$. Note that $J_h$ is a graded ideal and, if $l'$ is odd, then $l=l'$ and $I_h=\tilde I_h$.

\begin{Lem}\label{coideal}
The ideals $I_h$, $\tilde I_h$ and $J_h$ are all coideals of $A_q(n)$.
\end{Lem}
\begin{proof}
The assertion for $J_h$ and $I_h$ is well known, using \cite[(3.4)]{DPW}. More precisely, we have
\begin{equation*}
\aligned
\Dt(c_{i,j}^{lp^{h-1}})&=\sum_{1\leq k\leq n}c_{i,k}^{lp^{h-1}}\ot c_{k,j}^{lp^{h-1}}\in J_h\ot A_q(n)+A_q(n)\ot J_h,\quad (i\neq j)\\
\Dt(c_{i,i}^{lp^{h-1}}-1)&=\sum_{k\not=i}c_{i,k}^{lp^{h-1}}
\ot c_{k,i}^{lp^{h-1}}+(c_{i,i}^{lp^{h-1}}-1)\ot c_{i,i}
^{lp^{h-1}}+1\ot(c_{i,i}^{lp^{h-1}}-1)\\
&\in I_h\ot A_q(n)+A_q(n)\ot I_h.
\endaligned
\end{equation*}
If $l'$ is odd, then $\tilde I_h=I_h$ is the coideal of $A_q(n)$.
Now we assume $l'$ is even. Then
\begin{equation*}
\begin{split}
\Dt(c_{i,i}^{l'p^{h-1}}-1)&=(\Dt(c_{i,i}^{lp^{h-1}}))^2-1\ot 1\\
&=\bigg(\sum_{1\leq k\leq n}c_{i,k}^{lp^{h-1}}
\ot c_{k,i}^{lp^{h-1}}\bigg)^2-1\ot 1\\
&=\sum_{j\not=k}c_{i,j}^{lp^{h-1}}c_{i,k}^{lp^{h-1}}\ot
c_{j,i}^{lp^{h-1}} c_{k,i}^{lp^{h-1}}+\sum_{1\leq k\leq n}c_{i,k}^{l'p^{h-1}}\ot c_{k,i}^{l'p^{h-1}}-1\ot 1\\
&=\sum_{j\not=k}c_{i,j}^{lp^{h-1}}c_{i,k}^{lp^{h-1}}\ot
c_{j,i}^{lp^{h-1}} c_{k,i}^{lp^{h-1}}+
\sum_{k\not=i}c_{i,k}^{l'p^{h-1}}\ot c_{k,i}^{l'p^{h-1}}\\
&\qquad+(c_{i,i}^{l'p^{h-1}}-1)\ot
c_{i,i}^{l'p^{h-1}}+1\ot(c_{i,i}^{l'p^{h-1}}-1)\\
&\in \tilde I_h\ot A_q(n)+A_q(n)\ot \tilde I_h.
\end{split}
\end{equation*}
Thus, $ \tilde I_h $ is a coideal of $A_q(n)$.
\end{proof}
Now, by the above lemma,
$A_q(n)/J_h$, $A_q(n)/I_h$ and $A_q(n)/ \tilde I_h $ are all
bialgebras, and $A_q(n)/J_h$ is graded.
Let $A_q(n,r)_h$ be the
subspace of $A_q(n)/J_h$ consisting of the homogeneous polynomials
of degree $r$ in the $c_{ij}$. Since $A_q(n,r)_h$ is a finite
dimensional subcoalgebra of $A_q(n)/J_h$, its dual
$$s_k(n,r)_h=A_q(n,r)_h^*$$ is a finite dimensional algebra, which
is called an {\it infinitesimal $q$-Schur algebra} in \cite{Cox} and
\cite{Cox00} (cf. \cite{DNP1}). There are two canonical maps
\begin{equation}\label{tpi}
\pi:A_q(n)/J_h\twoheadrightarrow A_q(n)/I_h\text{ and
}\tilde\pi:A_q(n)/J_h\twoheadrightarrow A_q(n)/ \tilde I_h .
\end{equation}
Since $\pi(A_q(n,r)_h)$
and $\tilde\pi(A_q(n,r)_h)$ are all coalgebras, we may
define the algebras
$$u_k(n,r)_h=(\pi(A_q(n,r)_h))^*,\quad \tu_k(n,r)_h=(\tilde\pi(A_q(n,r)_h))^*.$$
 By definition, we see easily that
\begin{equation}\label{sandwich}
u_k(n,r)_h\han\tu_k(n,r)_h\han s_k(n,r)_h.
\end{equation} In the case of
$l'$ being an odd number, we have $u_k(n,r)_h=\tu_k(n,r)_h$. In general, we will use these inclusions together with results on simple  modules of $u_k(n,r)_1$ and $s_k(n,r)_1$, which is stated in the next theorem, to determine all simple  $\tu_k(n,r)_1$-modules in \S5.

\begin{Rem} When $l'$ is even, the coideal $\tilde I_h$ was not introduced in the literature, say, \cite{DD} or \cite{Cox, Cox00}. The definitions of $u_k(n,r)_h$ and $s_k(n,r)_h$ are independent of $l'$, while that of $\tu_k(n,r)$ depends on $l'$.
We will establish below in \ref{the relation between the little q-Schur algebra with the
infinitesimal q-Schur algebra} a connection between  $\tu_k(n,r)_1$ and the little $q$-Schur algebra $\tu_k(n,r)$.
\end{Rem}

Let
$D_q=\sum_{\pi\in\frak
S_n}(-1)^{\ell(\pi)}c_{1,1\pi}c_{2,2\pi}\cdots c_{n,n\pi}\in A_q(n)$ be the quantum determinant, where $\frak S_n$ is the symmetric
group and $\ell(\pi)$ is the length of $\pi$. Then
the localization $A_q(n)_{D_q}$ is a Hopf
algebra. Let $G=G_q(n)$ be the quantum linear group whose
coordinate algebra is $k[G]:=A_q(n)_{D_q}$.
Following \cite[\S 3.1,\S3.2]{Don} (see also \cite[1.3]{Cox} and \cite{Cox00}), let $G_h$ be the
$h$-th Frobenius kernel and $G_hT$, where $T=T_q(n)$ be the torus of $G$, be the corresponding ``Jantzen
subgroups''. Then
\begin{equation}\label{Gh=Mh}
k[G_h]:=A_q(n)_{D_q}/\langle I_h\rangle\cong A_q(n)/I_h\quad\text{ and }\quad
k[G_hT]:=A_q(n)_{D_q}/\langle J_h\rangle,
\end{equation}
and $A_q(n)/J_h$ is the polynomial part of $k[G_hT]$.\footnote{If one introduces the quantum matrix semigroup $M$, its `torus' $D$ and the $h$th Frobenoius kernel $M_h$, then $A_q(n)$, $A_q(n)/I_h$, $A_q(n)/J_h$ are respectively the coordinate algebras of $M$, $M_h$ and $M_hD$.}

Denote the character group of $T$ by $$\sfX:=\mathbb Z^n\cong X(T).$$
For each $\la\in \sfX$, by
\cite[3.1(13)(i)]{Don} (see also \cite[1.7]{Cox} and
\cite{Cox00}),  there is a
simple object $L_h(\la)$ in the category Mod($G_h$) of $G_h$-modules and a simple object $\h
L_h(\la)$ in the category  Mod$(G_hT)$ of $G_hT$-modules. Let
$$\sfX_h:=X_h(T)=\{\la\in \sfX(=\mathbb Z^n)\mid0\leq \la_i-\la_{i+1}\leq lp^{h-1}-1,\ 1\leq i\leq n \},$$
where we set $\la_{n+1}=0$. In particular, $\sfX_1=\{\la\in \mathbb Z^n\mid0\leq\la_i-\la_{i+1}<l, 1\leq i\leq n \},$
\begin{Thm}$($\cite[3.1(13),(18)]{Don}$)$
\label{the relation between various irreducible module} The set
$\{L_h(\la)\mid\la\in \sfX_h\}$ is a full set of nonisomorphic
simple   $G_h$-modules, and $\{\h L_h(\la)\mid\la\in \sfX\}$ is a
full set of nonsiomorphic simple  $G_hT$-modules. Moreover, for
all $\la\in \sfX$, we have $\h L_h(\la)|_{G_h}\cong L_h(\la)$.
\end{Thm}

By \cite{Cox},\cite{Cox00} (cf. \cite{DNP1}),
every polynomial $G_hT$-module (equivalently, $A_q(n)/J_h$-comodule) $V$ has a direct sum decomposition
$V=\oplus_{r\geq 0}V_r$, where $V_r$ is the $r$th homogeneous component (i.e., is an $s_k(n,r)_h$-module).
In particular,  if $|\la|=r$, then $\h L_h(\la)$ is an $s_k(n,r)_h$-module.
Note that a $G_h$-module  does not have such a direct sum
decomposition, since the decomposition $A_q(n)/I_h=\sum_{r\geq
0}\pi(A_q(n,r)_h)$ is not direct sum. However, if $|\la|=r$, then $L_h(\la)$ is a $u_k(n,r)_h$-module.

We now relate the $q$-Schur algebras to quantum enveloping algebra of $\mathfrak{gl}_n$ as given in \cite{BLM}, and define little $q$-Schur algebras.

Let $\Xi^{\pm}(n)$ be the set of all $A\in\Xi(n)$ whose diagonal
entries are zero. Given $r>0$, $A\in\Xi^\pm(n)$ and
${\bf j}=(j_1,j_2,\cdots,j_n)\in \mbz^n$, we define
\begin{equation*}
A({\bf j},r)=\sum_{D\in\Xi^0(n) \atop
\sigma(A+D)=r}\upsilon^{\sum_id_ij_i}[A+D]\in\bfU(n,r):=U_{\mathbb Q(\up)}(n,r).
\end{equation*}
where $\Xi^0(n)$ denotes the subset of diagonal matrices in
$\Xi(n)$ and $D=\text{diag}(d_1,\cdots,d_n)$.

The following result follows from \cite[5.5,5.7]{BLM} (see also \cite[(5.7)]{DPW}, \cite[A.1]{Du92} and \cite[3.4]{Du96}).
For $1\leq i,j\leq
n$, let $E_{i,j}\in\Xi(n)$ be the matrix unit $(a_{k,l})$ with
$a_{k,l}=\delta_{i,k}\delta_{j,l}$.

\begin{Thm} There is an algebra epimorphism $\zeta_r:\bfU(n)\twoheadrightarrow\bfU(n,r)$ satisfying
$$E_h\map E_{h,h+1}(\mathbf 0,r),\ K_1^{j_1}K_2^{j_2}\cdots
K_n^{j_n}\mapsto 0(\mathbf j,r),\ F_h\map E_{h+1,h}(\mathbf
0,r).$$
Moreover, $\zeta_r(U_\sZ(n))=U_\sZ(n,r)$ (\cite{Du}).
\end{Thm}

 For $A\in\Xi(n)$ and for
$\bft=(t_1,\cdots,t_n)\in\mbn^n$, let $\left[{\ttk_i;c \atop
t_i}\right]=\zr\left(\left[{K_i;c \atop t_i}\right]\right)$,
$\ttk_\bft=\prod\limits_{i=1}^n\left[{\ttk_i;0 \atop t_i}\right]$.
 Let
$$\tte_i=\zr(E_i),\quad
\ttf_i=\zr(F_i),\quad \ttk_j=\zr(K_j)\quad\text{ for }1\leq i\leq n-1, 1\leq
j\leq n.$$
 Let $U_\sZ^+(n,r)$ (resp., $U_\sZ^-(n,r),$ $U_\sZ^0(n,r)$)
be the $\sZ$-subalgebras of $U_\sZ(n,r)$ generated by the
$\tte_i^{(m)}$ (resp., $\ttf_i^{(m)}$, $\ttk_\la$), where $1\leq i\leq n-1$ and
$\la\in\La(n,r)=\{\la\in\mbn^n\mid \sg(\la)=r\}$. Here, $\sg(\la)=\la_1+\cdots+\la_n$.

\begin{Lem}$($\cite{DG2,DP}$)$\label{the property of Ur}
$(1)$ The set $\{\ttk_\la\ \mid \ \la\in\La(n,r)\}$ is a complete set
of primitive orthogonal idempotents $($hence a basis$)$ for
$U_\sZ^0(n,r)$. In particular, $1=\sum_{\la\in\La(n,r)}\ttk_\la.$

$(2)$ Let $\la\in\La(n,r)$, then
$\ttk_i\ttk_\la=\up^{\la_i}\ttk_\la$ for $1\leq i\leq n$.
\end{Lem}
Since $U_k(n,r)\cong U_\sZ(n,r)\ot_\sZ k$, $\zr$ naturally induces a
surjective homomorphism $\zr\ot 1:U_k(n)\twoheadrightarrow
U_k(n,r).$ For convenience, we shall denote $\zr\ot 1$ by $\zr$.
Similarly, we denote $\tte_i\ot 1$, $\ttf_i\ot 1$, $\ttk_j\ot 1$
by $\tte_i$, $\ttf_i$, $\ttk_j$.

The algebra $\tu_k(n,r):=\zr(\tu_k(n))$ is called a {\it little $q$-Schur algebra} in
\cite{DFW,Fu07} and is generated by $\tte_i$, $\ttf_i$, $\ttk_j$. Putting $\tu_k^+(n,r)=\zr(\tu_k^+(n))$,
$\tu_k^-(n,r)=\zr(\tu_k^-(n))$ and $\tu_k^0(n,r)=\zr(\tu_k^0(n))$, we have
$\tu_k(n,r)=\tu_k^-(n,r)\tu_k^0(n,r)\tu_k^+(n,r).$
Let
$$s_k(n,r)=\tu_k^-(n,r)U_k^0(n,r)\tu_k^+(n,r).$$ This is the subalgebra of $U_k(n,r)$ generated by the
elements $\tte_i,\ \ttf_i,\ \ttk_j,\ \left[{\ttk_j;0 \atop
t}\right]$ $(1\leq i\leq n-1,\ 1\leq j\leq n,\ t\in\mbn)$. We shall see below that $s_k(n,r)$ is isomorphic to the infinitesimal $q$-Schur algebra $s_k(n,r)_1$.

\begin{Rem}\label{unr} When $l'=l$ is odd, the restriction $\zeta_r:\tu_k(n)\to\tu_k(n,r)$  factors through the
quotient algebra $u_k(n)$, the infinitesimal quantum $\mathfrak{gl}_n$, defined at the end of \S2. Thus,
in this case, $\tu_k(n,r)$ is the same algebra as $u_k(n,r)$ considered in \cite{DFW}.
\end{Rem}

For a positive integer $m$, let
$\mbz_{m}=\mbz/m\mbz$.
Let
$$\ba{(\,\,)}_m :\mbz^n\ra(\mbz_{m})^n$$ be
the map defined by
$\ol{(j_1,j_2,\cdots,j_n)}=(\ol{j_1},\ol{j_2},\cdots,\ol{j_n}).$
For a subset $Y$ of $\mbz^n$, we shall denote $\ol{Y}_{m}=\{\bar
y\in(\mbz_{m})^n\mid y\in Y\}$.

For $\ol\la\in(\mbz_{l'})^n$, define
\begin{equation*}\ttp_{\lb}=
\begin{cases}
\sum\limits_{\mu\in\La(n,r),\mb=\lb}\ttk_\mu &
 \text{if}\ \lb\in\lrb\\
 0 & \text{otherwise}.
\end{cases}
\end{equation*}
\begin{Lem}\label{the property of ur}\cite{DFW,Fu07}
The set $\{\ttp_{\lb}\mid\lb\in\lrb\}$ forms a $k$-basis of
$\tu_k^0(n,r)$.
\end{Lem}

Let $\Xi(n)_h$ be the set of all $A=(a_{ij})\in\Xi(n)$ such that
$a_{ij}<lp^{h-1}$ for all $i\neq j$. Let $\Xi(n)_h^\pm=\{A\in\Xi(n)_h\mid a_{i,i}=0,\,\forall i\}$. Let $\Xi'(n)_h$ be the set
of all $n\times n$ matrices $A=(a_{ij})$ with $a_{ij}\in\mbn$,
$a_{ij}<lp^{h-1}$ for all $i\neq j$ and $a_{ii}\in\mbz_{l'p^{h-1}}$
for all $i$. We have an obvious map $pr:\Xi(n)_h\ra\Xi'(n)_h$
defined by reducing the diagonal entries modulo $l'p^{h-1}$. We
denote
$\Xi(n,r)_h:=\big\{A\in\Xi(n)_h\mid\sg(A)=r\big\}$ and $\Xi(n,r)_h^\pm=\Xi(n,r)_h\cap \Xi(n)_h^\pm$.

 Clearly, by regarding $s_k(n,r)_h$ as a subalgebra of the $q$-Schur algebra $U_k(n,r)$, the set
$$\{[A]\mid A\in\Xi(n,r)_h\}$$
forms a $k$-basis for $s_k(n,r)_h$.

Assume $A\in\Xi(n)_h^\pm$ with $\sg(A)\leq r$. Given $\lb\in\ol{\La(n,r-\sg(A))}_{l'p^{h-1}}$, let
\begin{equation}\label{[[A,r]]}
[\![A+\diag(\lb),r]\!]_h=\sum_{\mu\in\La(n,r-\s(A)) \atop
\mb=\lb}[A+\diag(\mu)].\end{equation}

\begin{Lem}\label{basis-tuk(n,r)}
The set $$\{[\![A+\diag(\lb),r]\!]_h\mid A\in\Xi(n,r)_h^\pm,\,\lb\in\ol{\La(n,r-\sg(A))}_{l'p^{h-1}}\}$$
forms a $k$-basis for $\tu_k(n,r)_h$. Thus
$\dim_k\tu_k(n,r)_h=\#pr(\Xi(n,r)_h)$. Similarly, the
set
$$\bigg\{\sum_{\mu\in\La(n,r-\s(A)) \atop
\mb=\lb}\xi_{A+\diag(\mu)}\mid
A\in\Xi(n,r)_h^\pm,\,\lb\in\ol{\La(n,r-\sg(A))}_{lp^{h-1}}\bigg\}$$
forms a $k$-basis for $u_k(n,r)_h$.
\end{Lem}
\begin{proof}
By \cite[4.2.4]{Fu05}, the set
\begin{equation}\label{cA}
\{c^{A+\diag(\lb)}+ \tilde I_h \mid
A\in\Xi(n,r)_h^\pm,\lb\in\ol{\La(n,r-\sg(A))}_{l'p^{h-1}}\}
\end{equation}
forms a $k$-basis for $\tilde\pi(A_q(n,r)_h)$. Similar to \cite[5.5.3]{Fu05},\footnote{The argument given in \cite{Fu05} is for the quantum coordinate algebra $A_{q,1}(n)$, while $A_q(n)$ considered here is $A_{1,q}(n)$. Here $A_{\alpha,\beta}(n)$ is the two parameter version defined in \cite{T1}.} we have
\begin{equation}\label{cA^*}
(c^{A+\diag(\lb)}+ \tilde I_h )^*=\sum_{\mu\in\La(n,r-\s(A)) \atop
\mb=\lb}\xi_{A+\diag(\mu)}=\vs^{d_{A+\diag(\la)}}[\![A+\diag(\lb),r]\!]_h.
\end{equation}
Here baring on $\mu$ is relative to $l'p^{h-1}$.
The first assertion follows. Replacing $l'p^{h-1}$ by $lp^{h-1}$ in the bar map, the first quality of \eqref{cA^*} gives the second assertion.
\end{proof}

The proof above gives immediately
the following result. This result, in terms of two parameter quantum linear groups, is the $(1,q)$-version of \cite[5.5]{Fu05} which is the $(q,1)$-version, see the footnote above.
\begin{Coro}
\label{the relation between the little q-Schur algebra with the
infinitesimal q-Schur algebra}We have algebra isomorphisms $\tu_k(n,r)_1\cong \tu_k(n,r)$ and  $s_k(n,r)_1\cong s_k(n,r)$.
\end{Coro}
Note that $u_k(n,r)$ (see \ref{unr}) is not defined when $l'$ is even. However, $u_k(n,r)_1$ is always defined, regardless $l'$ is odd or even.


\section{The classification of simple  modules of little  $q$-Schur algebras}

In this section, we shall  give the classification of simple  modules  for the little $q$-Schur algebra $\tu_k(n,r)$.



For $\lb\in(\mbz_{l'})^n$, let
$\fM_k(\lb)=
\tu_k(n)/\fI_k(\lb)$
where
$$
\fI_k(\lb)=
\sum_{1\leq i\leq n-1}
\tu_k(n)E_i+\sum_{1\leq i\leq n}\tu_k(n)(K_i-\vs^{\la_i}).$$
Then $\fM_k(\lb)$ has a unique irreducible quotient, which will be denoted by $\fL_k(\lb)$(see the proof of \cite[5.11]{L2}).
If $l'$ is odd, then $K_i^{l'}$ is central in $\tu_k(n)$ and hence $\fL_k(\lb)$ can be regarded as a $u_k(n)$-module.


Let
$\La^+(n,r)=\{\la\in\La(n,r)\mid \la_1\geq\la_2\geq\cdots\geq\la_n\}$ and let
$\La^+(n)=\cup_{r\geq 0}\La^+(n,r)$. For
$\la\in\La^+(n,r)$ let $V(\la)$ be the simple
$\bfU(n,r)$-module with highest weight $\la$. Let $x_\la$ be the highest weight vector of
$V(\la)$. Let $V_\sZ(\la)=U_\sZ(n,r)x_\la$. Since $U_\sZ(n,r)$ is a homomorphic image $U_\sZ(\frak{sl_n})$ by \ref{the image of UA(sln)}, we have $V_\sZ(\la)=U_\sZ(\frak{sl_n})x_\la$.
We denote $V_k(\la)=V_\sZ(\la)\ot_\sZ k$ and
let $L_k(\la)$ be the unique irreducible quotient of $V_k(\la)$.
For convenience, we shall denote the image of $x_\la$ in
$V_k(\la)$ and $L_k(\la)$ by the same letter. Let
$V_k'(\la)=\tilde u_k(n)x_\la$. We call $V_k'(\la)$ the
baby Weyl module of $\tilde u_k(n)$. Then $\fL_k(\lb)$ is the unique
irreducible quotient of $V_k'(\la)$.

If $\la\in \sfX_1$, then, by \ref{the little weyl
module of u'(g)}, and \ref{Lk=Lk'},
$\tu_k(\frak{sl}_n)x_\la=V_k(\la)$ and $L_k(\la)|_{\tu_k(\frak{sl}_n)}$ is irreducible. This together with
$\tu_k(\frak{sl}_n)\han \tu_k(\frak{gl}_n)\han U_k(n)$ implies the following.

  \begin{Lem}\label{res Lk(la)} For any $\la\in \sfX_1$, we have $V_k'(\la)=V_k(\la)$ and restriction gives $\tu_k(n)$-module isomorphisms
$\ L_k(\la)|_{\tu_k(n)}\cong\fL_k(\lb)$.
\end{Lem}

Note further that
$\ol{(\sfX_1)}_{l'}\han\ol{\La^+(n)}_{l'}=(\mbz_{l'})^n$ and
\begin{equation}\label{odd case}
\ol{(\sfX_1)}_{l'}=\ol{\La^+(n)}_{l'}=(\mbz_{l'})^n,\,\,\text{ if $l'$ is odd.}
\end{equation}

\begin{Thm}[{\cite[5.11,6.6]{L2}}]\label{the irreducible module of u(n)}
(1) If $l'$ is odd, then $l'=l$ and the set
$\{\fL_k(\nu)\mid \nu\in(\mbz_{l})^n\}$ forms a
complete set of non-isomorphic simple  $u_k(n)$-modules.

(2) If
$l'$ is even, then $l'=2l$ and the set
$\{\fL_k(\nu)\mid \nu\in(\mbz_{2l})^n\}$ forms a
complete set of non-isomorphic simple  $\tu_k(n)$-modules.
\end{Thm}

Note that, unlike the classification for simple  $G_1$-modules given in Theorem \ref{the relation between various irreducible module}, this classification depends on $l'$. We will make a comparison in Corollary \ref{missed}.

By \ref{res Lk(la)} and \eqref{odd case},  if $l'$ is odd, then every simple  $\tu_k(n)$-module on which all $K_i^l$ act as the identity is a $u_k(n)$-module and is also a restriction of a simple  $U_k(n)$-module with a restricted highest weight. However, when $l'$ is even, there are simple  $\tu_k(n)$-modules which cannot be realized in this way.

\begin{Bsp}\label{epl} Assume $l'=4$. Then $l=2$.
Let $V(2,0)=\bfU(2)/I(2,0)$,
where $$I(2,0)=\bfU(2))E_1+\bfU(2)F_1^{3}+\bfU(2))( K_1-\up^{2}))+\bfU(2))(K_2-1)).$$
Let $x_0=1+I(2,0)\in V(2,0)$. Let $V_\sZ(2,0)$ be the
$U_\sZ(2)$-submodule of $V(2,0)$ generated by $x_0$. Let
$V_k(2,0)=V_\sZ(2,0)\ot_\sZ k$.  Let $V'_k(2,0)=\tilde u_k(2)x_0$. The set $\{x_0 , F_1x_0, F_1^{(2)}x_0\}$ forms a $k$-basis for $V_k(2,0)$. Since $F_1^2=0\in U_k(2) $, $V_k'(2,0)=\text{span}_k\{x_0,F_1x_0\}$. Thus $V_k'(2,0)\not=V_k(2,0)$. Since $E_1(F_1x_0)=0$, $\text{span}_k\{F_1x_0\}$ is a submodule of $V_k'(2,0)$ (resp.,
$V_k(2,0)$). Hence, $L_k'(2,0)$ is one dimensional, while dim$\,L_k(2,0)=2$.

By regarding $L_k'(2,0)$ as the $\tu_k(\mathfrak {sl}_2)$-module $L_k'(2)$, we see that there is no simple
$U_k(\mathfrak {sl}_2)$-module $L_k(m)$ such that $L_k(m)|_{\tu_k(\mathfrak {sl}_2)}\cong L_k'(2)$.

\end{Bsp}
We are now ready to classify simple  $\tilde u_k(n,r)$-modules.

\begin{Lem}\label{the lemma for irreducible module of u(n,r)}
Let $L$ be a $\tilde u_k(n,r)$-module. Assume $x_0\neq 0\in L$
satisfies $k_ix_0=\ep^{\la_i}x_0$ for some $\la_i\in\mbn$ with $1\leq i\leq n$. Then
$\lb=(\ol{\la}_1,\cdots,\ol{\la}_n)\in\lrb$.
\end{Lem}
\begin{proof}
By \ref{the property of Ur} and \ref{the property of ur}, we have
$1=\sum_{\alpha\in\lrb}\ttp_{\alpha}$ and $\ttp_{\alpha}\in\tilde
u_k(n,r)$. It follows that $x_0=\sum_{\alpha\in\lrb}(\ttp_{\alpha}x_0)$
and hence, there exist $\beta\in\lrb$ such that
$\ttp_{\beta}x_0\not=0$. By \ref{the property of Ur},
$$\ep^{\la_i}x_0=\ttk_ix_0=\ttk_i\sum_{\alpha\in\lrb}\ttp_{\alpha}x_0
=\sum_{\alpha\in\lrb}\ep^{\alpha_i}\ttp_{\alpha}x_0$$ for $1\leq i\leq n$.
Hence,
$$\ep^{\la_i}\ttp_{\beta}x_0=\ttp_{\beta}(\ep^{\la_i}x_0)=
\ttp_{\beta}\sum_{\alpha\in\lrb}\ep^{\alpha_i}\ttp_{\alpha}x_0=
\ep^{\beta_i}\ttp_{\beta}x_0$$ for $1\leq i\leq n$. Since
$\ttp_{\beta}x_0\not=0$, we have $\ep^{\la_i}=\ep^{\beta_i}$ for $1\leq
i\leq n$ and hence, $\lb=\beta\in\lrb$.
\end{proof}

Let
$$\sfX_h(l)=\sfX_h+l\mbn^n\quad\text{ and }\quad
\sfX_h(l,r)=\{\la\in \sfX_h(l)\mid \,\sg(\la)=r\}.$$ For $h=1$ and
$\la\in \sfX_1(l,r)$, the irreducible (polynomial) $G_1T$-module $\h L_1(\la)$ given in \ref{the relation between various
irreducible module} is in fact an irreducible $A_q(n,r)_1$-comodule. Hence, $\h L_1(\la)$ has a
natural $s_k(n,r)_1$-module structure.


\begin{Thm}\label{the classification of irreducible module of u(n,r)}
For $\la\in \sfX_1(l,r)$ we have $\h L_1(\la)|_{\tilde
u_k(n,r)}\cong \fL_k(\lb)$. Moreover the set
$\{\fL_k(\lb)\mid \lb\in\ol{ \sfX_1(l,r)}_{l'}\}$ forms a complete set of
non-isomorphic simple $\tilde u_k(n,r)$-modules.
\end{Thm}
\begin{proof}
By \cite{Cox,Cox00} (cf. \ref{the relation between various irreducible module}),
\begin{equation}\label{the irreducible module of s(n,r)}
{\text{the set }\{\h
L_1(\la)\mid \la\in \sfX_1(l,r)\}\text{ forms a complete set}}\atop {\text{ of
non-isomorphic simple  $s_k(n,r)_1$-modules}.}
\end{equation}
 Thus, it is enough to prove that for each $\la\in \sfX_1(l,r)$, $\h L_1(\la)|_{\tu_k(n,r)}$ is irreducible, and every irreducible $\tu_k(n,r)$-module is isomorphic to $\fL_k(\mb)$ for some $\mb\in\ol{ \sfX_1(l,r)}_{l'}$

By \ref{the relation between various
irreducible module}, for $\la\in \sfX_1(l,r)$, we see that $\h
L_1(\la)|_{G_1}$ is a simple  $G_1$-module at level $r$. Hence, by \eqref{Gh=Mh}, $\h
L_1(\la)|_{{u_k(n,r)}_1}$ is a simple  $u_k(n,r)_1$-module.
Now the inclusions $u_k(n,r)_1\han \tilde u_k(n,r)\han s_k(n,r)_1=s_k(n,r)$ given in \eqref{sandwich}
force that $\h L_1(\la)|_{\tilde u_k(n,r)}$ is a simple $\tu_k(n,r)$-module. Hence, inflation gives a simple $\tu_k(n)$-module. Since $\h L_1(\la)$ is a highest  weight $s_k(n,r)$-module,
 by \cite[5.10(b)]{L2}, there exists $x_0\in\h L_1(\la)$ such that $E_ix_0=0$ and $K_ix_0=\varepsilon^{\la_i}x_0$ for all $i$. Now, the argument in \cite[5.11]{L2} implies that $\h L_1(\la)|_{\tilde u_k(n,r)}$ is isomorphic to $\fL_k(\lb)$.

On the other hand,
let $L$ be a simple  $\tilde u_k(n,r)$-module, then, by inflation, $L$ is a simple  $\tilde u_k(n)$-module. (If $l'$ is an odd
number, $L$ is also a simple  $u_k(n)$-module.) Hence,
there is some $x_0\neq 0\in
L$ such that $E_ix_0=0$ and $k_jx_0=\ep^{\la_j}x_0$ for $1\leq
i\leq n-1,1\leq j\leq n$, where
$\la=(\la_1,\cdots,\la_n)\in\mbn^n$. By \ref{the lemma for
irreducible module of u(n,r)}, we have $\bar\la\in\lrb$. So, without loss, we may choose
 $\la\in\La(n,r)$. We consider the
$s_k(n,r)$-module $s_k(n,r)\ot_{\tu_k(n,r)}L$. Let $V$ be the
$s_k(n,r)$-submodule of $s_k(n,r)\ot_{\tu_k(n,r)}L$ generated by
$\ttk_\la\ot x_0$. Then $V=s_k(n,r)(\ttk_\la\ot
x_0)=\tu_k^-(n,r)(\ttk_\la\ot x_0)$. It is clear that the
$\ttk_\la\ot x_0$ is the highest weight vector of $V$. Hence there
is a unique maximal $s_k(n,r)$-submodule of $V$, say $V_{max}$.
Let $L'=V/V_{max}$. Then $L'$ is a simple  $s_k(n,r)$-module
and hence $L'\cong\h L_1(\la)$ by \eqref{the irreducible module of
s(n,r)}. Since $\tte_i.(\ttk_\la\ot x_0)=0$ and
$\ttk_j.(\ttk_\la\ot x_0)=\ep^{\la_j}(\ttk_\la\ot x_0)$, 
by \ref{the irreducible module of
u(n)}, we have $L\cong\h L_1(\la)|_{\tu_k(n,r)}\cong \fL_k(\lb)$.
The proof is completed.
\end{proof}

\begin{Coro}
Every simple  $u_k(n)$-module when $l'=l$ is odd (resp. every simple  $\tu_k(n)$-module when $l'$ is even) is an inflation of a simple   $\tu_k(n,r)$-module for some r. \end{Coro}
\begin{proof}
Since $\mbz^n=\sfX_1(T)+l\mbz^n$, it follows that
$\bin_{r\geq 0}\ol{\sfX_1(l,r)}_{l'}=\ol{\sfX_1(l)}_{l'}=\ol{(\sfX_1)}_{l'}+\ol{(l\mbn^n)}_{l'}
=\ol{(\sfX_1)}_{l'}+\ol{(l\mbz^n)}_{l'}=\mbz_{l'}^n$. Thus, by \ref{the classification of irreducible module of u(n,r)}, inflation via epimorphisms $\tu_k(n)\to\tu_k(n,r)$ gives simple  $\tu_k(n)$-modules indexed by $\mbz_{l'}^n$. Now the assertion follows from \ref{the irreducible module of u(n)}.
\end{proof}

We remark that restricted simple $U_k(n)$-modules $L_k(\la)$ with $\la\in\sfX_1$ does not cover all simple  $\tu_k(n)$-module when $l'$ is even. The above result shows that simple  $G_1T$-modules {\it does} cover all simple $\tu_k(n)$-module.

\begin{Coro} \label{missed}We have, for $\la\in\sfX_1$ with $\sigma(\la)=r$, $L_1(\la)\cong\fL_k(\lb)|_{u_k(n,r)_1}$ where $\lb\in\ol{(\sfX_1)}_{l'}$. In other words, $\{\fL_k(\nu)\mid\nu\in\ol{(\sfX_1)}_{l'}\}$ is a complete set of all simple  $G_1$-modules.
\end{Coro}
Note that this classification is the same as the one given in \ref{the relation between various irreducible module} since  the set $\ol{(\sfX_1)}_{l'}$ can be identified with $\sfX_1$ via the map
$\sfX_1\to \ol{(\sfX_1)}_{l'}, \la\mapsto \lb.$
Thus, for the example $L_k'(2,0)$ constructed in \ref{epl}, its restriction to  $G_1$ is again irreducible and isomorphic to $L_1(0,0)$.

\begin{Rems} \label{l vz l'}
(1) If $\la\in\sfX_1$, \ref{res Lk(la)} and \ref{the classification of irreducible module of u(n,r)} imply that restriction induces isomorphisms $L_k(\la)\cong \h L_1(\la)$, $\h L_1(\la)\cong\fL_k(\lb)$  and $\fL_k(\lb)\cong L_1(\la)$.

(2) When $l'=l$ is odd, we established in \cite{DFW} that a basis for $\tu_k(n,r)$ is indexed by $\ol{\Xi(n,r)}_l$, while a basis for $U_k(n,r)$ is indexed by $\Xi(n,r)$. Similarly,
since $\ol{\La^+(n,r)}_{l}=\ol{ \sfX_1(l,r)}_{l}$, simple $\tu_k(n,r)$-modules are indexed by
 $\ol{\La^+(n,r)}_l$, while simple $U_k(n,r)$-modules are indexed by
 $\La^+(n,r)$. Thus, baring on the index sets gives the counterparts for little  $q$-Schur algebras. However, if $l'$ is even, then $\ol{\La^+(n,r)}_{l'}$ is even {\it not} a subset of $\ol{ \sfX_1(l,r)}_{l'}$ and the classification is quite different.
\end{Rems}

Note that, as a comparison, \eqref{the irreducible module of s(n,r)} shows that the classification for the infinitesimal $q$-Schur algebra is independent of $l'$.

\section{The baby transfer map}
There is an epimorphism $\psi_{r+n,r}:U_k(n,r+n)\twoheadrightarrow U_k(n,r)$, called the transfer map in \cite[\S2]{L00}. This map can be geometrically constructed (\cite{Groj} and \cite[\S2]{L00}) and algebraically constructed by quantum coordinate algebras and quantum determinant (\cite[5.4]{Du95}). Since $\psi_{r+n,r}$ satisfies
$$\zeta_{r+n}(E_i)\longmapsto \zeta_{r}(E_i),\quad\zeta_{r+n}(F_i)\longmapsto \zeta_{r}(F_i),\quad
\zeta_{r+n}(K_i)\longmapsto \ep\zeta_{r}(K_i),$$
its restriction induces an epimorphism
\begin{equation}\label{transfer map}
\psi_{r+n,r}:\tu_k(n,r+n)\twoheadrightarrow \tu_k(n,r).
\end{equation}

In this section, we introduce the
{\it baby transfer map} $\rhol:\tu_k(n,r+l')\ra \tu_k(n,r)$ and use it to prove that, up to isomorphism, there only finitely many little $q$-Schur algebras. By these maps, we will understand
the classification of simple  modules for the little $q$-Schur
algebras from a different angle.

\begin{Prop}\label{the map rhol}
There is an algebra epimorphism
$\rhol:\tu_k(n,r+l')\twoheadrightarrow \tu_k(n,r)$ satisfying
$$\tte_i'\map\tte_i,\quad \ttf_i'\map\ttf_i,\quad \ttk_j'\map\ttk_j,$$
where $\tte_i',\ttf_i',\ttk_i'$ are the corresponding $\tte_i,\ttf_i,\ttk_i$ for $\tu_k(n,r+l')$.
Moreover, for $A\in\Xi^\pm(n)_1$ and $\lb\in\ol{\La(n,r+l'-\sg(A))}_{l'}$,
\begin{equation}\label{the element rhol([[A,r]])}
\rhol([\![A+\diag(\lb),r+l']\!])=
\begin{cases} [\![A+\diag(\lb),r]\!]& \text{if } \lb\in\ol{\La(n,r-\sg(A))}_{l'}\\
0 &\text{otherwise}.
\end{cases}
\end{equation}
\end{Prop}
\begin{proof}Consider the epimorphism $\tilde\pi:A_q(n)/J_1\to A_q(n)/\tilde I_1$ given in \eqref{tpi}. Since every monomial $m$ in $A_{q}(n,r)_1$ has the same homomorphic image as the monomial $c_{11}^{l'}m\in A_{q}(n,r+l')_1$, it follows that $\tilde\pi(A_{q}(n,r)_1)\han\tilde\pi({A_{q}(n,r+l')_1})$ and
the basis $\big\{c^{A+\diag(\la)}+\tilde I_1\bm| A\in\Xi^\pm(n)_1,\,\lb\in\ol{\La(n,r-\sg(A))}_{l'}\big\}$ for $\tilde\pi (A_{q}(n,r)_1)$ given in \eqref{cA} extends to a basis for $\tilde\pi(A_{q}(n,r+l')_1)$. By taking dual and \ref{the
relation between the little q-Schur algebra with the infinitesimal
q-Schur algebra}, there is an algebra epimorphism
$\rhol:\tu_k(n,r+l')\ra
\tu_k(n,r)$.

Let
$\big\{(c^{A+\diag(\la)}+\tilde I_1)^*\bm| A\in\Xi^\pm(n)_1,\,\lb\in\ol{\La(n,r'-\sg(A))}_{l'}\big\}$
be the dual basis for
$\tu_k(n,r')$. It is now clear that, for
$A\in\Xi^\pm(n)_1$ and $\lb\in\ol{\La(n,r+l'-\sg(A))}_{l'}$,
\begin{equation*}\rhol((c^{A+\diag(\lb)}+\tilde I_1)^*)=
\begin{cases}
(c^{A+\diag(\lb)}+\tilde I_1)^* & \text{if } \lb\in\ol{\La(n,r-\sg(A))}_{l'}\\
0 & \text{otherwise}.
\end{cases}
\end{equation*} This together with  \eqref{cA^*} implies
\eqref{the element rhol([[A,r]])} since $\ep$ is a primitive
$l'$-th root of unity. Since $\ttp_{\lb}=[\![\diag(\lb),r]\!]$, we
have for $\lb\in\ol{\La(n,r+l')}_{l'}$
\begin{equation*}\rhol({\ttp_{\lb}})=
\begin{cases}
\ttp_{\lb}& \text{if } \lb\in\lrb\\
0& \text{otherwise}.
\end{cases}
\end{equation*}
Hence,
$\rhol(\ttk_i')=\rhol(\sum_{\lb\in\ol{\La(n,r+l')}_{l'}}
\ep^{\la_i}\ttp_{\lb})=\sum_{\lb\in\lrb}\ep^{\la_i}\ttp_{\lb}=\ttk_i$.
Similarly, we can prove $\rhol(\tte_i')=\tte_i$ and
$\rhol(\ttf_i')=\ttf_i$.
\end{proof}

Observe that, if $r\geq (n-1)(l'-1)$, then, for $\la\in\La(n,r+l')$, $\sum_{1\leq i\leq n}\la_i=r+l'\geq n(l'-1)+1$. Thus, $\la_i\geq l'$ for some $i$. Consequently,
$\ol\la=\ol{(\la_1,\cdots,\la_{i}-l',\cdots,\la_n)}\in\ol{\La(n,r)}_{l'}$. Thus, we see that
\begin{equation}\label{Lem}
\ol{\La(n,r)}_{l'}=\ol{\La(n,r+l')}_{l'} \text{ whenever }r\geq (n-1)(l'-1).
\end{equation}

\begin{Coro}\label{Cor1}
We have $\ti u_k(n,r)\cong\ti u_k(n,r+l')$ for $r\geq (l-1)(n^2-n)+(n-1)(l'-1)$. Hence, up to isomorphism,
there are only finitely many little $q$-Schur algebras.
\end{Coro}
\begin{proof}
By \ref{the map rhol}, there exists an algebra epimorphism from $\ti u_k(n,r+l')$ to $\ti u_k(n,r)$. Thus, it is enough to prove that $\dim_k\ti u_k(n,r+l')=\dim_k\ti u_k(n,r)$ for $r\geq (l-1)(n^2-n)+(n-1)(l'-1)$. By \cite[8.2]{DFW} and \cite[6.8]{Fu07}, we have
\begin{equation}\label{dim}
\dim_k\tu_k(n,r)=|\{(A,\lb)\mid A\in\Xi^\pm(n)_1,\,\lb\in\ol{\La(n,r-\sg(A))}_{l'}\}|.
\end{equation}
If $r\geq (l-1)(n^2-n)+(n-1)(l'-1)$, then, for any $A\in\Xi^\pm(n)_1$,
$$
r-\sg(A)\geq r-(n^2-n)(l-1)\geq (n-1)(l'-1).
$$
Thus, by \eqref{Lem}, $\ol{\La(n,r-\sg(A))}_{l'}=\ol{\La(n,r-\sg(A)+l')}_{l'}$. Consequently,  \eqref{dim}, implies $\dim_k\ti u_k(n,r+l')=\dim_k\ti u_k(n,r)$ whenever $r\geq (l-1)(n^2-n)+(n-1)(l'-1)$. This completes the proof.
\end{proof}

We now look at the second application of the baby transfer map.
If $l'$ is odd, then $l'=l$ by definition and the index set of the classification given in \ref{the classification of irreducible module of u(n,r)} becomes
\begin{equation}\label{rem}
\aligned
\ol{ \sfX_1(l,r)}_{l}&=\{\lb\mid \la\in \sfX_1,\sg(\la)\leq r,\ol{\sg(\la)}=\ol{r}\}\\
&=\ol{\sfX_1(l,r-l)}_{l}\cup\{
\lb\mid\la\in \sfX_1,\,\sg(\la)=r
\}.\\\endaligned
\end{equation}
This indicates, by  \ref{the
classification of irreducible module of u(n,r)} and \ref{the map rhol}, that the simple
$\tu_k(n,r)$-modules can be divided into two classes, one consists of
the simple  $\tu_k(n,r)$-modules which can be obtained by restriction from the
simple  $U_k(n,r)$-modules with restricted highest weights and the other consists of the
simple  $\tu_k(n,r)$-module which are inflations of the
simple  $\tu_k(n,r-l)$-module via the map $\rho_{r,r-l}$. The disjointness of the two classes can be seen as follows.

Suppose $n\geq r$. Let $\og=(1^r)\in\La(n,r)$. Then $\ttk_\og=\ttp_\og\in\tu_k(n,r)$. By \cite[7.1]{Fu05} we have
$\ttk_\og\tu_k(n,r)\ttk_\og$ is isomorphic to the Hecke algebra $\sH_r=\sH(\fS_r)$. We will identify $\ttk_\og\tu_k(n,r)\ttk_\og$ with $\sH_r$. Thus, we may define the ``baby'' Schur functor $F_r$ as follows:
\begin{equation*}
F_r:\text{Mod}(\tu_k(n,r))\lra \text{Mod}(\sH_r),\quad
V\longmapsto\ttk_\og V.
\end{equation*}
The functor $F_r$ induces a group homomorphism over the Grothendieck groups
$$\ttF_r:K(\tu_k(n,r))\ra K(\sH_r).$$
Here $K(A)$ denotes the Grothendieck group of $\text{Mod}(A)$.

By \ref{the map rhol} the category $\text{Mod}(\tu_k(n,r-l'))$ can be regarded as a full subcategory of $\text{Mod}(\tu_k(n,r))$ via $\rhol$ and hence we may view $K(\tu_k(n,r-l'))$ as a subgroup of
$K(\tu_k(n,r))$.

\begin{Prop}
Assume $l'$ is odd and $n\geq r$. Then $\ttF_r$ is surjective and $\ker(\ttF_r)=
K(\tu_k(n,r-l))$.
\end{Prop}
\begin{proof}
By \ref{res Lk(la)} and \cite[4.4(2)]{Don},
the set $\{F_r(\fL_k(\lb))\mid\la\in \sfX_1,\,\sg(\la)=r\}$
forms a complete set of non-isomorphic simple  $\sH_r$-modules.
Thus by \cite[(6.2(g))]{Gr80} and \eqref{rem}, we conclude that
$F_r(\fL_k(\lb))=0$ for $\lb\in\ol{\sfX_1(l,r-l)}_{l}$. The assertion follows.
\end{proof}


\section{Semisimple little $q$-Schur algebras}
We now determine semisimple little $q$-Schur
algebras. This can be easily done by the semisimplicity of the infinitesimal $q$-Schur algebra $s_k(n,r)=s_k(n,r)_1$ and the following.

\begin{Lem}\label{socle}
Let $V$ be an $s_k(n,r)$-module. Then
$\soc_{s_k(n,r)}V=\soc_{\tu_k(n,r)}V$.
\end{Lem}
\begin{proof}
It is easy to check that $\soc_{s_k(n,r)_1}V=\soc_{G_1T}V$ and
$\soc_{u_k(n,r)_1}V=\soc_{G_1}V$. By \cite[3.1(18)(iii)]{Don}
we have $\soc_{G_1T}V=\soc_{G_1}V$. It follows that
$\soc_{s_k(n,r)_1}V=\soc_{u_k(n,r)_1}V$. Since $u_k(n,r)_1\han
\tu_k(n,r)\han s_k(n,r)_1$ the assertion follows from \ref{the
relation between various irreducible module} and \ref{the
classification of irreducible module of u(n,r)}.
\end{proof}


\begin{Thm}
The little $q$-Schur algebra $\tu_k(n,r)$ is semisimple if and only
if either $l>r$ or $l=n=2$ and $r\geq3$ is odd.
\end{Thm}
\begin{proof}
By \cite[1.2]{QF2}, the infinitesimal $q$-Schur algebra $s_k(n,r)$ is semisimple if and only
if either $l>r$ or $n=2$, $l=2$ and $r\geq3$ is odd. Thus,
it is enough to prove that
the infinitesimal $q$-Schur algebra $s_k(n,r)$ is semisimple if
and only if the little $q$-Schur algebra $\tu_k(n,r)$ is semisimple.

Suppose the algebra $s_k(n,r)$ is semisimple. Let $W$ be an
indecomposable projective $\tu_k(n,r)$-module. Since $s_k(n,r)$ is
semisimple, the $s_k(n,r)$-module $s_k(n,r)\ot_{\tu_k(n,r)}W$ is
semisimple. By \ref{the classification of irreducible module of
u(n,r)}, $(s_k(n,r)\ot_{\tu_k(n,r)}W)|_{\tu_k(n,r)}$ is a
semisimple $\tu_k(n,r)$-module. Since $W$ is a projective
$\tu_k(n,r)$-module, $W$ is a flat $\tu_k(n,r)$-module. It follows
that the natural $\tu_k(n,r)$-module homomorphism from $W\cong
\tu_k(n,r)\ot_{\tu_k(n,r)}W$ to $s_k(n,r)\ot_{\tu_k(n,r)}W$ is
injective and hence $W$ is a semisimple $\tu_k(n,r)$-module. So
the algebra $\tu_k(n,r)$ is semisimple.

Now we suppose the algebra $s_k(n,r)$ is not semisimple. Then
there exist $\la,\mu\in \sfX_1(l,r)$ such that $\Ext_{s_k(n,r)}(\h
L_1(\la),\h L_1(\mu))\not=0$ and hence, there exists an
$s_k(n,r)$-module $V$ such that $\soc_{s_k(n,r)}V=\h L_1(\mu)$
with top $\h L_1(\la)$. By \ref{socle} we have
$\soc_{\tu_k(n,r)}V=\h L_1(\mu)$ and hence $\tu_k(n,r)$ is not
semisimple.
\end{proof}

 We will see in the next section (at least when $l'$ is odd) that
 the semisimplicity of $\tu_k(n,r)$ depends only on $r$ and $l$, while the infinitesimal quantum group $\tu_k(n)$ is never semisimple (for all $n$ and $l'=l$).

\section{Little $q$-Schur algebras of finite representation type}

In this section, we will assume $k$ is an algebraically closed field and $l'$ is odd. Thus, $l'=l\geq3$ and $\ol{\sfX_1(l,r)}_l=\ol{\La^+(n,r)}_l$ (see \ref{l vz l'}(2)). By \ref{the relation between various irreducible module} and \ref{the classification of irreducible module of u(n,r)}, $L_1(\la)=\fL_k(\lb)$ for all $\la\in\sfX_1(l,r)$.
We will classify little $q$-Schur algebras of finite representation type in this case.  The even case is much more complicated and will be treated elsewhere.

We first determine the blocks of $\tu_k(2,r)$. Using it, we then establish that $\tu_k(2,r)$ has finite representation type if and only if it is semisimple. We then generalize this from $n=2$ to an arbitrary $n$.

Blocks of $q$-Schur
algebras were classified in \cite{Cox,Cox98} (cf. also \cite{D94}). Moreover, blocks of
infinitesimal $q$-Schur algebras were classified in \cite{Cox,Cox00}
for $n=2$. Now we first classify blocks of little $q$-Schur algebras $\tu_k(2,r)$ and use this
to determine their finite representation type.


Let $\Phi=\{\be_i-\be_j\mid 1\leq i\neq j\leq n\}$ be the set of roots of
type $A_{n-1}$ where
$$\be_i=(0,\cdots,0,\underset i1,0\cdots,0)\in\mathbb Z^n.$$
Let $\Phi^+=\{\be_i-\be_j\mid 1\leq i<j\leq n\}$ be the set of
positive roots. There is a $\mbz$-bilinear form
$\langle-,-\rangle$ on $X=\mbz^n$ satisfying
$\langle\be_i,\be_j\rangle=\dt_{ij}$ for $1\leq i,j\leq n$. The
symmetric group $\frak S_n$ acts on $X$ by place permutation.
The `dot' action of $\frak S_n$ on $X$ is defined as:
$w_\cdot\la=w(\la+\rho)-\rho$ where $\rho=(n-1,n-2,\cdots,1,0)$. For
$\la\in\La^+(n,r)$, let $m(\la)$ be the least positive integer $m$ such
that there exists an $\al\in\Phi^+$ with
$\langle\la+\rho,\al\rangle\not\in lp^{m}\mbz$.

\begin{Prop}$($\cite{Cox,Cox98}$)$\label{block1}
For $\la\in\La^+(n,r)$, let $\sB^{n,r}(\la)$ be the block of $q$-Schur algebras $U_k(n,r)$
containing $L_k(\la)$. Then, we have
$$\sB^{n,r}(\la)=({\frak
S}_{n}\!\,_\cdot\la+lp^{m(\la)}\mbz\Phi)\cap\La^+(n,r).$$
\end{Prop}

We will denote the block of $G_h$
containing $L_h(\la)$ by $\sB_h^n(\la)$ for $\la\in \sfX_h$ and denote the
block of infinitesimal $q$-Schur algebras $s_k(n,r)_h$ containing
$\hl_1(\la)$ by $\sB_h^{n,r}(\la)$ for $\la\in\sfX_h(l,r)$.

\begin{Prop}$($\cite{Cox,Cox00}$)$\label{block2}
Assume $n=2$. For $\la\in \sfX_h$, we have
$$\sB_h^2(\la)=({\frak
S}_{2}\!\,_\cdot\la+lp^{m(\la)}\mbz\Phi+lp^{h-1}\sfX)\cap\sfX_h.$$ For
$\la\in\sfX_h(l,r)$ we have
\begin{equation*}
\sB_h^{2,r}(\la)=
\begin{cases}
({\frak
S}_{2}\!\,_\cdot\la+lp^{m(\la)}\mbz\Phi)\cap\sfX_h(l,r)&\text{ if
$m(\la)+1\leq h$},\\
\{\la\}&\text{ if $m(\la)+1>h$}.
\end{cases}
\end{equation*}
\end{Prop}

For
$\bar\la\in\ol{\La^+(n,r)}_l$, the block of little $q$-Schur algebras $\tu_k(n,r)$
containing $\fL_k(\lb)$ will be denoted by ${\mathfrak b}^{n,r}(\lb)$. We now determine ${\mathfrak b}^{2,r}(\lb)$.

\begin{Lem}\label{symmetric}
For any $\lb,\mb\in\ol{ \sfX_1(l,r)}_l$, we have
$$\Ext_{\tu_k(n,r)}^1(\fL_k(\lb),\fL_k(\mb))
\cong\Ext_{\tu_k(n,r)}^1(\fL_k(\mb),\fL_k(\lb)).$$
\end{Lem}
\begin{proof}
By \cite[3.10]{BLM}, there is an anti-automorphism $\tau$ on the
$q$-Schur algebra $U_k(n,r)$ by sending $[A]$ to $[^t\!A]$ for all
$A\in\Xi(n,r)$, where $^t\!A$ is the transpose of $A$. Since
the set $\{[\![A,r]\!]_1\mid A\in\ol{\Xi(n,r)}_l\}$ (see \eqref{[[A,r]]}) forms a $k$-basis of
$\tu_k(n,r)$ by \cite{DFW}, we conclude that  $\tau(\tu_k(n,r))=\tu_k(n,r)$.
Using $\tau$, we may construct, for any (finite dimensional) $\tu_k(n,r)$-module $M$, its
contravariant dual module $^\tau\! M$. Thus, as a vector space, $^\tau M$ is the dual space $M^*$ of $M$ and the action is defined by $x.f=f\tau(x)$ for all $x\in\tu_k(n,r),f\in M^*$.
Since $^\tau\!(\fL_k(\lb))\cong
\fL_k(\lb)$ for any $\lb\in\ol{ \sfX_1(l,r)}_l$ and $0\to L\to M\to N\to0$ is an exact sequence
of $\tu_k(n,r)$-modules if and only if so is $0\to {}^\tau\!N\to{}^\tau\! M\to{}^\tau\! L\to0$,
the result follows easily (see \cite[II, 2.12(4)]{Jan87} for a similar result.)
\end{proof}

\begin{Prop}\label{block}
For $\lb\in\ol{\La^+(2,r)}_l(=\ol{ \sfX_1(l,r)}_l)$ with $\la\in\La^+(2,r)$, if ${\mathfrak b}^{2,r}(\lb)$ denotes the block containing $\fL_k(\lb)$ for the little $q$-Schur algebra $\tu_k(2,r)$, then
 $${\mathfrak b}^{2,r}(\lb)=\ol{({\frak
S}_{2}\!\,_\cdot\la)}_l\cap\ol{\La^+(2,r)}_l=\ol{\sB^{2,r}(\la)}_l
=\ol{\sB_1^{2,r}(\la)}_l.$$
\end{Prop}
\begin{proof}
If
$\mu\in \sfX_1(l,r)$ and $\Ext^1_{s_k(2,r)_1}(\h
L_1(\la),\h L_1(\mu))\not=0$, then, by \ref{socle},
$\Ext^1_{\tu_k(2,r)}(\fL_k(\lb),\fL_k(\mb))\not=0$. This proves
$\ol{\sB_1^{2,r}(\la)}_l\han{\mathfrak b}^{2,r}(\lb)$. Hence, \ref{block2} implies
$\ol{({\frak S}_{2}\!\,_\cdot\la)}_l\cap\ol{\La^+(2,r)}_l\han{\mathfrak b}^{2,r}(\lb)$.

On the other hand, if $\mb\in{\mathfrak b}^{2,r}(\lb)$ with $\mu\in \sfX_1(l,r)$ and $\Ext_{\tu_k(2,r)}^1(\fL_k(\lb),\fL_k(\mb))\not=0$, then there
exists a $\tu_k(2,r)$-module $N$ (and hence a $G_1$-module) such that
$\soc_{\tu_k(2,r)}N=\fL_k(\mb)$ and $\text{top}_{\tu_k(2,r)}N=\fL_k(\lb)$. Equivalently, as a $G_1$-module,
$\soc_{G_1}N=\soc_{u_k(2,r)_1}N=L_1(\mu)$ and hence
$\text{top}_{G_1}N=L_1(\la)$. So
$\Ext_{G_1}^1(L_1(\la),L_1(\mu))\not=0$. Thus, the first assertion in \ref{block2} implies
$\mu\in ({\frak
S}_{2}\!\,_\cdot\la+lp^{m(\la)}\mbz\Phi+l\sfX)\cap\sfX_1$. Since $\ol{\sfX_1(l,r)}_l=\ol{\La^+(2,r)}_l$, it follows from \ref{the
classification of irreducible module of u(n,r)} that ${\mathfrak b}^{2,r}(\lb)\han\ol{({\frak
S}_{2}\!\,_\cdot\la)}_l\cap\ol{\La^+(2,r)}_l$.   Hence, ${\mathfrak b}^{2,r}(\lb)=\ol{({\frak
S}_{2}\!\,_\cdot\la)}_l\cap\ol{\La^+(2,r)}_l$, and consequently,
${\mathfrak b}^{2,r}(\lb)=\ol{\sB^{2,r}(\la)}_l=\ol{\sB_1^{2,r}(\la)}_l$, by \ref{block1}
and \ref{block2}.
\end{proof}

We are now going to establish the fact that any non-semisimple $\tu_k(2,r)$ has infinite representation type. We need the following three simple lemmas.
\begin{Lem}\label{indc}
Let $V$ be an $s_k(n,r)$-module. Then, $V$ is an indecomposable
$s_k(n,r)$-module if and only if $V$ is an indecomposable
$\tu_k(n,r)$-module.
\end{Lem}
\begin{proof}
It is clear that $V$ is an indecomposable $s_k(n,r)_1$-module
(respectively, $u_k(n,r)_1$-module) if and only if $V$ is an
indecomposable $G_1T$-module (respectively, $G_1$-module). By
\cite[3.1(18)]{Don}, $V$ is an indecomposable $G_1T$-module
if and only if $V$ is an indecomposable $G_1$-module. The
assertion now follows from \eqref{sandwich}.
\end{proof}

\begin{Lem}[{\cite[3.4(2)]{Fu08}}]\label{rbe}
Let $N$ be an $U_k(n,r)$-module with two composition factors
$L_k(\la)$ and $L_k(\mu)$, where $\la\in \sfX_h$ and
$\mu\in\La^+(n,r)$ with $\soc_{U_k(n,r)}N\cong L_k(\la)$. Assume
that $L_k(\mu)=\oplus_{j=1}^s\hlr(\mu_j)$ is the decomposition of
$L_k(\mu)$ into irreducible $s_k(n,r)_h$-modules. If $\hlr(\la)\ncong\hlr(\mu_j)$
as $G_h$-modules for all $j$, then $\soc_{s_k(n,r)_h}N\cong
L_k(\la)\cong\hlr(\la)$.
\end{Lem}


\begin{Lem}\label{decomposition of Ae}
Let $A$ be finite dimensional $k$-algebra and $e$ is an idempotent element in $A$. Assume $\{L_i\mid i\in I\}$ is a complete set of non-isomorphic irreducible $A$-modules. Then we have
$$Ae\cong\bigoplus_{i\in I}\dim_k(eL_i)P(L_i),$$ where $P(L_i)$ is the projective cover of $L_i$. In particular, if $l$ is odd and $A=U_k(2,r)$ with $r=l$ or $l+1$, then $U_k(2,l)\ttk_{(l-1,1)}\cong P(l-1,1)$, $U_k(2,l)\ttk_{(l,0)}\cong U_k(2,l)\ttk_{(0,l)}\cong P(l,0)$, and
$U_k(2,l+1)\ttk_{(l-1,2)}\cong P(l-1,2)\oplus
P(l,1)$, $U_k(2,l+1)\ttk_{(l+1,0)}\oplus
U_k(2,l+1)\ttk_{(1,l)}\cong 2P(l+1,0)\oplus P(l,1)$.
\end{Lem}
\begin{proof}
Since $e$ is an idempotent element in $A$, $Ae$ is projective and hence we may write $$Ae\cong\bigoplus_{i\in I}d_iP(L_i),$$
where $d_i\in\mbn$. Then
$\dim_k(eL_i)=\dim_k\Hom_A(Ae,L_i)=\sum_{j\in I}d_j\dim_k\Hom_A(P_j,L_i)=d_i$ for $i\in I$.

The last statement follows from the following facts:  If $A=U_k(2,l)$, then there are $\frac{l-1}2+1$ simple modules:
$L_k(l-i,i)$ ($0\leq i\leq \frac{l-1}2$). For each $1\leq i\leq \frac{l-1}2$, $L_k(l-i,i)$ has dimension $l-2i+1$ and weights
$(l-i-j,i+j)$ with $0\leq j\leq l-2i$, while $L_k(l,0)$ has dimension 2 and  weights $(l,0)$ and $(0,l)$ by the tensor product theorem.
If $A=U_k(2,l+1)$, then there are $\frac{l+1}2+1$ simple modules:
$L_k(l+1-i,i)$ ($0\leq i\leq \frac{l+1}2$). For each $1\leq i\leq \frac{l+1}2$, $L_k(l+1-i,i)$ has dimension $l-2i+2$ and weights
$(l+1-i-j,i+j)$ with $0\leq j\leq l-2i+1$, while $L_k(l+1,0)$ has dimension 4 and weights $(l+1,0),(l,1), (1,l)$ and $(0,l+1)$.
\end{proof}

For $\la\in\La^+(n,r)$ let $P(\la)$ be the projective cover of
$L_k(\la)$ as a $U_k(n,r)$-module. For $\la\in\La^+(n,r)$ let
$\mathfrak{p}(\bar\la)$ be the projective cover of $\fL_k(\lb)$ as a
$\tu_k(n,r)$-module.

\begin{Prop}\label{irt1}
The algebra $\tu_k(2,l)$ has infinite representation type.
\end{Prop}
\begin{proof}
Let $\la=(l,0)$ and $\mu=(l-1,1)$.   By \cite{Th}, the standard module
$\Delta(\la)$ has two composition factors with socle $L_k(\mu)$.
Since $U_k(2,r)$ is semisimple for $l>r$ (see, e.g., \cite{EN}), we have, for any
$\nu=(\nu_1,\nu_2)\in\La^+(2,l)$ with $\nu\neq\la$,
$\Delta(\nu)\cong\Delta(\nu_1-\nu_2,0)\ot\det_q^{\nu_2}\cong
L_k(\nu_1-\nu_2,0)\ot\det_q^{\nu_2}\cong L_k(\nu)$.
Hence, by the Brauer-Humphreys reciprocity, $P(\mu)={\Delta(\mu)\atop\Delta(\la)}$ and $P(\la)=\Delta(\la)$ are
uniserial modules with composition series
\begin{equation}\label{eq1-rep}
\begin{array} {cccccc}
 P(\mu):& L_k(\mu) &\qquad  P(\la):&L_k(\la)&\qquad  P(\nu):&L_k(\nu)\\
 &L_k(\la)&  &L_k(\mu)&&\\
&L_k(\mu) & & & &\\
 \end{array}
\end{equation}
where $\nu\not=\la,\mu$. By \ref{block}, we have
${\mathfrak b}^{2,r}(\lb)=\{\lb,\mb\}$ and ${\mathfrak b}^{2,r}(\bar\nu)=\{\nb\}$ for
$\nu\not=\la,\mu$.
Using \ref{decomposition of Ae}, we see that  $\tu_k(2,l)\ttp_{\mb}=U_k(2,l)\ttk_{\mu}\cong P(\mu)$. Hence,
$P(\mu)|_{\tu_k(2,l)}$ is a projective $\tu_k(2,l)$-module.
By \ref{rbe}, we first have $\soc_{s_k(2,l)}\Delta(\la)=\h L_1(\mu)$. Applying \ref{rbe} again to the contravariant dual of $P(\mu)/L_k(\mu)$
(see the proof of \ref{symmetric}) yields $\soc_{s_k(2,l)}(P(\mu)/L_k(\mu))=\h L_1(\la)$. Hence,
$\soc_{s_k(2,l)}P(\mu)$ is irreducible and hence $P(\mu)|_{s_k(2,l)}$ is indecomposable. This together with \ref{indc} implies that $P(\mu)|_{\tu_k(2,l)}$ is indecomposable. Thus, $P(\mu)|_{\tu_k(2,l)}\cong\mathfrak{p}(\bar\mu)$ has the following structure:
$$\begin{array} {cc}
 \mathfrak{p}(\bar\mu):& \fL_k(\bar\mu) \\
 &2\fL_k(\bar\la)\\
&\fL_k(\bar\mu)  \\
 \end{array}$$
Here $2\fL_k(\bar\la)$ means $\fL_k(\bar\la)\oplus \fL_k(\bar\la)$. (Note that $\fL_k(\bar\la)\cong\fL_k(0,0)$.)

Now let us determine the structure of  $\mathfrak{p}(\bar\la)$.
By \ref{decomposition of Ae} we have
\begin{equation}\label{eq2-rep}
U_k(2,l)\ttk_\la\cong U_k(2,l)\ttk_\dt\cong P(\la)
\end{equation}
where $\dt=(0,l)$. Let
\begin{equation*}
\begin{split}
W_1&=\text{span}\left\{
\bigg[\bigg(\begin{matrix}
a_{1,1}&0\\
a_{2,1}&0
\end{matrix}\bigg)\bigg]
\bigg|1\leq\,a_{1,1},a_{2,1}\leq l-1,\,a_{1,1}+a_{2,1}=l\right\}\\
W_2&=\text{span}\left\{
\bigg[\bigg(\begin{matrix}
0&a_{1,2}\\
0&a_{2,2}
\end{matrix}\bigg)\bigg]
\bigg|1\leq\,a_{1,2},a_{2,2}\leq l-1,\,a_{1,2}+a_{2,2}=l\right\}.
\end{split}
\end{equation*}
Then, there are vector space decompositions:
\begin{equation}\label{eq3-rep}
U_k(2,l)\ttk_\la=W_1\oplus\text{span}
\{
\big[\big(\begin{smallmatrix}
l& 0\\
0&0
\end{smallmatrix}\big)\big],
\big[\big(\begin{smallmatrix}
0& 0\\
l&0
\end{smallmatrix}\big)\big]
\},\quad
U_k(2,l)\ttk_\dt=W_2\oplus\text{span}
\{
\big[\big(\begin{smallmatrix}
0& l\\
0&0
\end{smallmatrix}\big)\big],
\big[\big(\begin{smallmatrix}
0& 0\\
0&l
\end{smallmatrix}\big)\big]
\},\quad
\end{equation}
Clearly, $\dim_k L_k(\la)=2$ and $L_k(\la)$ has only two weights $(l,0)$ and $(0,l)$. Thus, by \eqref{eq1-rep}, \eqref{eq2-rep} and \eqref{eq3-rep},
\begin{equation}\label{eq4-rep}
W_1\cong W_2\cong L_k(\mu).
\end{equation}
Now, as a vector space, $\tu_k(2,l)\ttp_{\lb}=W_1\oplus W_2\oplus\text{span}\{\ttp_{\lb}\}$.
Furthermore, by \ref{decomposition of Ae}, $\mathfrak{p}(\bar\la)\cong \tu_k(2,l)\ttp_{\lb}$.
Thus, by \eqref{eq4-rep}, $\text{soc}_{\tu_k(2,l)}\tu_k(2,l)\ttp_{\lb}=W_1\oplus W_2\cong 2\fL_k(\bar\mu)$ and
$\tu_k(2,l)\ttp_{\lb}/(W_1\oplus W_2)\cong \fL_k(\bar\la)$. So
$\mathfrak{p}(\bar\la)\cong \tu_k(2,l)\ttp_{\lb}$ has the following structure:
$$\begin{array} {cc}
 \mathfrak{p}(\bar\la):& \fL_k(\bar\la) \\
 &2\fL_k(\bar\mu).\\
 \end{array}$$

Let $B$ be the basic
algebra of the block ${\mathfrak b}^{2,r}(\lb)$ of $\tu_k(2,l)$. Let
$v_0=\fL_k(\bar\la)$ and $v_1=\fL_k(\bar\mu)$. The $\Ext$ quiver for  $B$ is
given by Figure 1:

\begin{figure}[h]
\begin{center}
\setlength{\unitlength}{.5cm}
\begin{picture}(18,2.9)
 \put(6.9,1.9){$\bullet$} \put(6.5,1.4){$v_{0}$}
\put(10.9,1.9){$\bullet$} \put(11.2,1.4){$v_{1}$}
  \put(7.4,2.2){\vector(1,0){3.2}}
   \put(7.4,2.9){\vector(1,0){3.2}}
\put(9,3.1){$\alpha_{1}$}  \put(9,2.4){$\alpha_{2}$}
\put(10.6,1.8){\vector(-1,0){3.2}} \put(9,1.1){$\beta_{1}$}
\put(10.6,0.9){\vector(-1,0){3.2}} \put(9,0.2){$\beta_{2}$}
\end{picture}
\caption{}
\end{center}
\end{figure}
\noindent
 with relations $\bt_1\al_1=\bt_2\al_2$ and $\bt_2\al_1=\bt_1\al_2=\al_i\bt_j=0$
for all $i,j\in\{0,1\}$.
Since $\mathfrak{p}(\bar\mu)$ is also an injective module and  $\mathfrak{p}(\bar\mu)$ is
the only indecomposable projective modules of radical length
greater than 2, by \cite[9.2]{DK} the algebra $B$ has infinite
representation type if and only if $B/J^2$ has infinite
representation type, where $J$ is the radical of $B$. Thus, by applying
\cite[11.8]{Pi} to the quiver above, we conclude that the algebra $B$ has infinite representation type.
Hence, the algebra $\tu_k(2,l)$ has infinite representation type.
\end{proof}

\begin{Prop}\label{irt2}
The algebra $\tu_k(2,l+1)$ has infinite representation type.
\end{Prop}
\begin{proof}
Let $\la=(l+1,0)$, $\mu=(l-1,2)$ and $\dt=(1,l)$. By the argument
similar to the proof of \ref{irt1} we have $P(\la)$ and $P(\mu)$
are uniserial modules with composition factors given by
$$\begin{array} {cccccc}
 P(\mu):& L_k(\mu) &\qquad  P(\la):&L_k(\la)&\qquad  P(\nu):&L_k(\nu)\\
 &L_k(\la)&  &L_k(\mu)&&\\
&L_k(\mu) & & & &\\
 \end{array}$$
where $\nu\not=\la,\mu$. By \ref{block}, we have
${\mathfrak b}^{2,r}(\lb)=\{\lb,\mb\}$ and ${\mathfrak b}^{2,r}(\bar\nu)=\{\nb\}$ for
$\nu\not=\la,\mu$. Applying \ref{decomposition of Ae} yields
 $$\tu_k(2,l+1)\ttp_{\mb}=U_k(2,l+1)\ttk_{\mu}\cong P(\mu)\oplus
P(\dt^+)\quad\text{ where }\quad\dt^+=(l,1).$$
So $P(\mu)|_{\tu_k(2,l+1)}$ is projective.
 A similar argument with \ref{rbe} as in the proof of \ref{irt1} shows that $\soc_{s_k(2,l+1)}P(\mu)$ is irreducible. Hence, $P(\mu)|_{s_k(2,l+1)}$ is indecomposable. Thus, by \ref{indc}, $P(\mu)|_{\tu_k(2,l+1)}$ is an indecomposable $\tu_k(2,l+1)$-module.
So, $P(\mu)|_{\tu_k(2,l+1)}\cong
\mathfrak{p}(\bar\mu)$. Now, by \ref{socle} and \ref{rbe}, $\mathfrak{p}(\bar\mu)$ has the following structure:
$$\begin{array} {cc}
 \mathfrak{p}(\bar\mu):& \fL_k(\bar\mu) \\
 &2\fL_k(\bar\la).\\
&\fL_k(\bar\mu)  \\
 \end{array}$$

We now determine the structure of $V:=\tu_k(2,l+1)\ttp_{\lb}$.
Let $W=U_k(2,l+1)\ttk_{\la}\oplus
U_k(2,l+1)\ttk_{\dt}$.  By \ref{decomposition of Ae}, noting that $L_k(\dt^+)$ is the Steinberg module,
\begin{equation}\label{eq5-rep}
\mathfrak{p}(\bar\la)\oplus \fL_k(\bar\dt^+)\cong V\han W\cong 2P(\la)\oplus L_k(\dt^+)
\end{equation}
So, by \ref{socle} and \ref{rbe}, $\soc_{\tu_k(2,l+1)} W\cong 2\fL_k(\bar\mu)\oplus
\fL_k(\bar\dt^+)$ and $W/\soc_{\tu_k(2,l+1)}W\cong 4\fL_k(\bar\la)$.
Thus there exist $\tu_k(2,l+1)$-submodules $W_1,W_2,W_3$ of $W$  such that $\soc_{\tu_k(2,l+1)}W=W_1\oplus W_2\oplus W_3$, $W_1\cong W_2\cong \fL_k(\bar\mu)$ and $W_3\cong \fL_k(\bar\dt^+)$.
Since $[\![\text{diag}(\bar\mu),r]\!]_1=[\text{diag}(\mu)]$, with the notation
$$V_{\mu}=[\![\text{diag}(\bar\mu),r]\!]_1V=[\text{diag}(\mu)]V\han [\text{diag}(\mu)]W= W_{\mu},$$
one computes $\dim_kV_{\mu}=\dim_k W_{\mu}=\dim_k(\soc_{\tu_k(2,l+1)}W)_{\mu}=3$. Thus, $V_{\mu}=W_{\mu}=(\soc_{\tu_k(2,l+1)}W)_{\mu}=(W_1)_{\mu}\oplus
(W_2)_{\mu}\oplus
(W_3)_{\mu}$. This implies $\soc_{\tu_k(2,l+1)}W=\tu_k(2,l+1)(W_1)_{\mu}\oplus
\tu_k(2,l+1)(W_2)_{\mu}\oplus \tu_k(2,l+1)
(W_3)_{\mu}\han V$. Hence, $\soc_{\tu_k(2,l+1)}V=\soc_{\tu_k(2,l+1)}
W\cong 2\fL_k(\bar\mu)\oplus \fL_k(\bar\dt^+)$. Since $\dim_kV/\soc_{\tu_k(2,l+1)}V=2=\dim_k \fL_k(\bar\la)$
and $V/\soc_{\tu_k(2,l+1)}V\han W/\soc_{\tu_k(2,l+1)}W\cong 4\fL_k(\bar\la)$, we have $V/\soc_{\tu_k(2,l+1)}
V\cong \fL_k(\bar\la)$. Thus, by \eqref{eq5-rep} $\mathfrak{p}(\bar\la)$ has
three composition factors with socle $2\fL_k(\bar\mu)$.

If $B$ denotes the
basic algebra of the block ${\mathfrak b}^{2,r}(\lb)$ of $\tu_k(2,l)$, then the computation above implies that the
$\Ext$ quiver for $B$ is the same as given in Figure 1 above with relations
$\bt_1\al_1=\bt_2\al_2$ and all other products are zero. Hence, $B$ has infinite representation type and, consequently, $\tu_k(2,l+1)$
has infinite representation type.
\end{proof}


We now can establish the follow classification of finite representation type for little $q$-Schur algebras.

\begin{Thm}\label{classification}
Assume $l'=l\geq3$ is odd. The little $q$-Schur algebra $\tu_k(n,r)=u_k(n,r)$ has finite representation
type if and only if $l>r$.
\end{Thm}
\begin{proof} Recall from \cite[8.2(2), 8.3]{DFW} that $u_k(n,r)$ has a basis $\{[\![A,r]\!]\}_{A\in\ol{\Xi(n,r)}_l}$. If $n>2$, then $e=\sum_{\la\in\ol{\La(2,r)}_l}[\![\diag(\la),r]\!]\in \tu_k(n,r)$ is an idempotent and $e\tu_k(n,r)e\cong \tu_k(2,r)$.
Thus, if $\tu_k(2,r)$ has infinite representation type, then so does $\tu_k(n,r)$ (see \cite{Bo} or \cite[I.4.7]{Erd90} for such a general fact).
So it reduces to prove the result for $n=2$.

If $r<l$, $\tu_k(n,r)=U_k(n,r)$ is semisimple by \cite{EN}. It remains to prove that $\tu_k(2,r)$ has infinite representation type for all $r\ge l$. By the transfer map \eqref{transfer map}, we see that either $\tu_k(2,l)$ or $\tu_k(2,l+1)$ is a homomorphic image of $\tu_k(2,r)$. Since both $\tu_k(2,l)$ and $\tu_k(2,l+1)$ have infinite representation type by \ref{irt1} and \ref{irt2}, it follows that the algebra $\tu_k(2,r)$ and hence, $\tu_k(n,r)$, has infinite representation type for all $r\geq l$.
\end{proof}

A by-product of this result is the following determination of finite representation type of infinitesimal quantum $\mathfrak{gl}_n$.

\begin{Coro}
The infinitesimal quantum group $u_k(n)$ has infinite representation type for any $n$ and $l$.
In particular, $u_k(n)$ is never semisimple.
\end{Coro}
\begin{proof}
By \ref{classification} the algebra $\tu_k(n,l)$ has infinite representation type. This implies that $u_k(n)$ has infinite representation type since $\tu_k(n,l)$ is the homomorphic image of $u_k(n)$.
\end{proof}

\vspace{.5cm}

\section{Appendix}

It is well known that $\zr(\bfU(\frak{sl}_n))$ is equal to
$\bfU(n,r)$. In this section, we shall prove that this is also
true over $\sZ$, that is, $\zr(U_\sZ(\frak{sl}_n))=U_\sZ(n,r)$.

Let $X_i:=\{\mu\in\La(n,r)\mid \text{max}\{\mu_j-\mu_{j+1}\ \mid 1\leq j
\leq n-1\}=i\}.$ Then we have $\La(n,r)=\bin\limits_{-r\leq i\leq
r}X_i$ (disjoint union). The definition of $\bfU(n)$ implies the
following result.
\begin{Lem}\label{the algebra automorphism sigma on U}
There is a unique $\mbq(\up)$-algebra automorphism $\s$ on
$\bfU(n)$ satisfying
$$\s(E_i)=F_i,\ \s(F_i)=E_i,\ \s(K_j)=K_j^{-1}.$$
\end{Lem}
It is clear that $\s(\big[{\ti{K}_i;c\atop
t}\big])=\big[{\ti{K}_i^{-1};c\atop t}\big]$. By definition, the
$\sZ$-algebra $U_\sZ(\frak{sl}_n)$ is generated by the elements
$E_i^{(N)}$, $F_i^{(N)}$ and $\ti{K}_i^{\pm 1}$ $(1\leq i\leq n,\
N\geq 0 )$. Since $\big[{\ti{K}_i;c\atop t}\big]\in U_\sZ(\frak
{sl}_n)$ and $\s(U_\sZ(\frak{sl}_n))=U_\sZ(\frak{sl}_n)$, we have
$\big[{\ti{K}_i^{-1};c\atop t}\big]\in U_\sZ(\frak{sl}_n)$. By
\ref{the property of Ur}(2), the following lemma holds in
$\bfU(n,r)$.
\begin{Lem}\label{the product formula in Ur}
Let $\la\in\La(n,r)$. Then we have $\big[{\ti{\ttk}_i; c\atop
t}\big]\ttk_\la=\big[{\la_i-\la_{i+1}+c \atop t}\big]\ttk_\la$.
\end{Lem}
\begin{Thm}\label{the image of UA(sln)}
The image of $U_\sZ(\frak{sl}_n)$ under the homomorphism $\zr$ is
equal to the algebra $U_\sZ(n,r)$. Hence, for any field $k$ which is a $\sZ$-algebra, base change
induces an epimorphism $\zeta_{r}=\zeta_r\otimes1:U_k(\frak{sl}_n)\to U_k(n,r)$.
\end{Thm}
\begin{proof}
Let $U_r'=\zr (U_\sZ(\frak{sl}_n))$. By \cite{Du},
$\zr(U_\sZ(n))=U_\sZ(n,r)$. Hence it is enough to prove that
$\ttk_\la\in U_r'$ for any $\la\in\La(n,r)$. We shall prove
$\ttk_{\mu}\in U_r'$ for any $\mu\in X_i$ by a downward induction on $i$.

It is clear that $X_r=\{\bla_i:=(0,\cdots,0,\underset
ir,0\cdots,0)\mid 1\leq i\leq n-1\}$ and
$X_{-r}=\{\bla_{n}:=(0,\cdots,0,r)\}$. By \ref{the property of
Ur}(1) and \ref{the product formula in Ur}, for $1\leq i\leq n-1$,
we have
$$\left[{\ti{\ttk}_i;r \atop
2r}\right]=\ttk_{\bla_i}+\sum_{\mu\in\La(n,r),\, \mu\neq
\bla_i}\left[{\mu_i-\mu_{i+1}+r \atop 2r}\right]\ttk_\mu.$$ If
$1\leq i\leq n-1$, then $0\leq \mu_i-\mu_{i+1}+r< 2r$ for any
$\mu\in\La(n,r)$ with $\mu\neq\bla_i$. Hence, $\left[{\mu_i-\mu_{i+1}+r \atop 2r}\right]=0$ and
$\big[{\ti{\ttk}_i;r
\atop 2r}\big]=\ttk_{\bla_i}\in U_r'$ for $1\leq i\leq n-1$.
Similarly, we can prove $\big[{\ti{\ttk}_{n-1}^{-1};r\atop
2r}\big]=\ttk_{\bla_{n}}\in U_r'.$ Hence for any $\mu\in X_r\cup
X_{-r}$ we have $\ttk_\mu\in U_r'$.

Now we assume that for any $\mu\in X_j$ with $j>k$ we have
$\ttk_\mu\in U_r'$. Let $\la\in X_k$. Then there exists some $\il$
such that $\la_{\il}-\la_{\ile}=k$. We need to prove $\ttk_\la\in
U_r'$.

By \ref{the property of Ur}(1) and \ref{the product formula in
Ur}, we have
$$\lk{\ti{\ttk}_\il;r \atop k+r} \rk=\sum_{\mu\in X_j,\, j\neq k}\lk{\mu_\il-\mu_\ile+r\atop k+r}\rk\ttk_\mu
+\sum_{\nu\in X_k}\lk{\nu_\il-\nu_\ile+r\atop k+r}\rk\ttk_\nu.$$
Note that for $j<k$ with $\mu\in X_j$, we have
$0\leq\mu_\il-\mu_\ile+r\leq j+r<k+r.$ Since
$0\leq\nu_\il-\nu_\ile+r\leq k+r$ for $\nu\in X_k$, we have
$0\leq\nu_\il-\nu_\ile+r< k+r$ where $\nu\in X_k$ such that
$\nu_\il-\nu_\ile\neq
 k$. It follows that
$$\lk{\ti{\ttk}_\il;r \atop k+r} \rk=\sum_{\nu\in X_k\atop\nu_\il-\nu_\ile=k}\ttk_\nu+\sum_{\mu\in X_j,\,
j>k}\lk{\mu_\il-\mu_\ile+1\atop k+1}\rk\ttk_\mu .$$ Let
$Z:=\{\nu\in X_k\mid \nu_\il-\nu_\ile=k\}$. Then by induction we have
\begin{equation}\label{the image (1) of UA(sln)}
\sum_{\nu\in Z}\ttk_\nu\in U_r'.
\end{equation}

For any $i\neq i_0$ and $-r\leq s\leq k$, let $Y_{s,i}:=\{\nu\in
Z\mid \nu_i-\nu_{i+1}=s\}$. Then for any fixed $i\neq i_0$, we have
$Z=\bin\limits_{-r\leq s\leq k}Y_{s,i}$ (disjointed union). Now
for fixed $i\neq i_0$, we prove $\sum_{\nu\in Y_{s,i}}\ttk_\nu\in
U_r'$ by induction on $s$.

For fixed $i\neq\il$, let $m:=\text{max}\{s\mid Y_{s,i}\neq \kong,\
-r\leq s\leq k\}$. By \ref{the property of Ur}(1) and \ref{the
product formula in Ur}, we have
\begin{equation*}
\begin{split}
\lk{\ti{\ttk}_i;r\atop m+r}\rk
&=\sum_{\mu\in\La(n,r)}\lk{\mu_i-\mu_{i+1}+r\atop m+r}\rk\ttk_\mu
\\&=\sum_{\nu\in Y_{m,i}}\ttk_\nu+\sum_{\nu\in Y_{s,i}\neq\kong\atop -r\leq
s<m}\lk{s+r \atop m+r}\rk\ttk_\nu+\sum_{\nu\not\in
Z}\lk{\nu_i-\nu_{i+1}+r\atop m+r}\rk\ttk_\nu\\
&=\sum_{\nu\in Y_{m,i}}\ttk_\nu+\sum_{\nu\not\in
Z}\lk{\nu_i-\nu_{i+1}+r\atop m+r}\rk\ttk_\nu\\
&(\text{since $0\leq s+r<m+r$ for $-r\leq s<m$}).
\end{split}
\end{equation*}
Hence, multiplying both sides by $\sum_{\nu\in Z}\ttk_\nu$, \eqref{the image (1) of UA(sln)} implies
$$\sum_{\nu\in Y_{m,i}}\ttk_\nu=\sum_{\nu\in Z}\ttk_\nu\lk{\ti{\ttk}_i;r\atop m+r}\rk\in U_r'.$$

Now we assume $Y_{s,i}\neq\kong$ and for any $s'$ such that $s'>s$
and $Y_{s',i}\neq\kong$ we have $\sum\limits_{\nu\in
Y_{s',i}}\ttk_\nu\in U_r'$. We now prove $\sum\limits_{\nu\in
Y_{s,i}}\ttk_\nu\in U_r'$.

By \ref{the property of Ur}(1) and \ref{the product formula in
Ur}, we have
\begin{equation*}
\begin{split}
\lk{\ti{\ttk}_i;r\atop s+r}\rk
&=\sum_{\mu\in\La(n,r)}\lk{\mu_i-\mu_{i+1}+r\atop
s+r}\rk\ttk_\mu\\
&=\sum_{\nu\in Y_{s,i}}\ttk_\nu+\sum_{\nu\in
Y_{s',i}\neq\kong\atop s<s'\leq m}\lk{s'+r \atop
s+r}\rk\ttk_\nu+\sum_{\nu\in Y_{s',i}\neq\kong\atop -r\leq
s'<s}\lk{s'+r \atop s+r}\rk\ttk_\nu +\sum_{\nu\not\in
Z}\lk{\nu_i-\nu_{i+1}+r\atop s+r}\rk\ttk_\nu\\
&=\sum_{\nu\in Y_{s,i}}\ttk_\nu+\sum_{\nu\in
Y_{s',i}\neq\kong\atop s<s'\leq m}\lk{s'+r \atop s+r}\rk\ttk_\nu
+\sum_{\nu\not\in Z}\lk{\nu_i-\nu_{i+1}+r\atop s+r}\rk\ttk_\nu.
\end{split}
\end{equation*}
By induction we have
$$\sum_{\nu\in Y_{s',i}\neq\kong\atop
s<s'\leq m}\lk{s'+r \atop s+r}\rk\ttk_\nu=\sum_{s<s'\leq
m}\lk{s'+r \atop s+r}\rk \sum_{\nu\in
Y_{s',i}\neq\kong}\ttk_\nu\in U_r'.$$ It follows that
$$\sum_{\nu\in Y_{s,i}}\ttk_\nu+\sum_{\nu\not\in
Z}\lk{\nu_i-\nu_{i+1}+r\atop
s+r}\rk\ttk_\nu=\lk{\ti{\ttk}_i;r\atop s+r}\rk-\sum_{\nu\in
Y_{s',i}\neq\kong\atop s<s'\leq m}\lk{s'+r \atop
s+r}\rk\ttk_\nu\in U_r'.$$ Hence by \eqref{the image (1) of
UA(sln)} we have
$$\sum_{\nu\in Y_{s,i}}\ttk_\nu=(\sum_{\nu\in Z}\ttk_\nu)
\cdot\biggl(\sum_{\nu\in Y_{s,i}}\ttk_\nu+\sum_{\nu\not\in
Z}\lk{\nu_i-\nu_{i+1}+r\atop s+r}\rk\ttk_\nu\biggr)\in U_r'.$$

Now we have proved $\sum\limits_{\nu\in
Y_{s,i}\neq\kong}\ttk_\nu\in U_r'$ for $i\neq i_0$ with $-r\leq
s\leq k$. It is clear that we have
\begin{equation*}
\begin{split}
\jiao\limits_{i\neq i_0 \atop 1\leq i\leq n-1
}Y_{\la_i-\la_{i+1},i}&=\{\nu\in
Z\mid \nu_i-\nu_{i+1}=\la_i-\la_{i+1},\ 1\leq i\leq n-1,\ i\neq
i_0\}\\
&=\{\nu\in X_k\mid \nu_i-\nu_{i+1}=\la_i-\la_{i+1},\ 1\leq i\leq
n-1\}\\
&=\{\la\}.
\end{split}
\end{equation*}
It follows that
\begin{equation*}
\begin{split}
\prod_{i\neq i_0\atop 1\leq i\leq n-1}\sum_{\nu\in
Y_{\la_i-\la_{i+1},i}}\ttk_\nu &=\sum_{\nu\in\jiao\limits_{i\neq
i_0,\,1\leq i\leq n-1
}Y_{\la_i-\la_{i+1},i}}\ttk_\nu\\
&=\ttk_\la\in U_r'.
\end{split}
\end{equation*}
Hence, the result follows.
\end{proof}

Note by the proof of the above theorem that we have in fact
proved $\zr(U_\sZ^0(\frak{sl}_n))=U_\sZ^0(n,r)$.
 It is natural to ask what is
the image of $\tu_k(\frak{sl}_n)$ under the map $\zr$. The following
theorem answer the question.
\begin{Thm}
If $(n,l')=1$, i.e., the integers $n$ and $l'$ are relatively prime,
then $\zr(\tu_k^0(\frak{sl}_n))=\tu_k^0(n,r)$. In particular, the
homomorphism $\zr:\tu_k(\frak{sl}_n)\ra \tu_k(n,r)$ is surjective.
\end{Thm}
\begin{proof}
Let $s=\zr(\tu_k(\frak{sl}_n)),\ s^+=\zr(\tu_k^+(\frak{sl}_n)), \
s^-=\zr(\tu_k^-(\frak{sl}_n))$ and $s^0=\zr(\tu_k^0(\frak{sl}_n))$.
Then $\tu_k^+(n,r)=s^+$ and $\tu_k^-(n,r)=s^-$. Hence it is enough to
prove $\ttk_i\in s^0$ for all $i$. Since
$\ttk_1\ttk_2\cdots\ttk_n=\ep^r$ by \cite[2.1]{DG2}, we have
$\ttk_1^n=\ep^r\ti{\ttk}_1^{n-1}
\ti{\ttk}_2^{n-2}\cdots\ti{\ttk}_{n-2}^2\ti{\ttk}_{n-1}\in s^0.$
Since $(n,l')=1$, there is some integers $a,b$ such that $na+bl'=1$.
So $\ttk_1=\ttk_1^{na+bl'}=\ttk_1^{na}\in s^0$. Then
$\ttk_1^{-1}=\ttk_1^{l'-1}\in s^0.$ Hence
$\ttk_{i+1}^{-1}=\ttk_1^{-1}\ti{\ttk}_1\ti{\ttk}_2\cdots\ti{\ttk}_i\in
s^0$ for all $i$. It follows $\ttk_{i+1}=\ttk_{i+1}^{-(l'-1)}\in
s^0$ for all $i$. Hence the result follows.
\end{proof}
\begin{Rem}Note that if $(n,l')\neq 1$, the above theorem may be
not true. For example, Suppose $n=l'=3=l$ and $r\geq 4$. Then
$\ti{\ttk}_2=\ttk_2\ttk_3^{-1}=\ep^{-r}\ttk_1\ttk_2^2$ since
$\ttk_1\ttk_2\ttk_3=\ep^r.$ Since $\ttk_2^3=\ttk_2^l=1$, we have
$\ti{\ttk}_1=\ttk_1\ttk_2^{-1}=\ttk_1\ttk_2^2$. Hence
$\ti{\ttk}_2=\ep^{-r}\ti{\ttk}_1.$ It follows that $\zr(\tu_k(\frak
{sl}_n)^0)=\text{span}\{1,\ti{\ttk}_1,\ti{\ttk_1}^2\}$. So
dim$\zr(\tu_k^0(\frak{sl}_n))\leq 3$. But, by \cite[9.2]{DFW},  dim\,$\tu_k^0(n,r)=9$.
Hence, in general, $\zr(\tu_k^0(\frak{sl}_n))\neq \tu_k^0(n,r).$
Thus,, it is very likely that $\zr:\tu_k(\frak{sl}_n)\ra \tu_k(n,r)$ is not surjective.
\end{Rem}


\begin{thebibliography}{}\frenchspacing

\bibitem{BLM}
A. A. Beilinson, G. Lusztig and R. MacPherson, \textit{A geometric
setting for the quantum deformation of $GL_n$}, Duke Math.J. {\bf
61} (1990), 655-677.

\bibitem{Bo} K. Bongartz, \textit{Zykellose hlgebren sind nicht zugellos}, Representation theory II,
Springer Lecture Notes in Mathematics {\bf 832} (1980), 97--102.

\bibitem{Cox}
A. G. Cox, \textit{On some applications of infinitesimal methods
to quantum groups and related algebras, Ph.D thesis}, University
of London {1997}.

\bibitem{Cox98}
A. G. Cox, The blocks of the $q$-Schur algebra, J. Algebra {\bf 207}
(1998), 306-325.

\bibitem{Cox00}
A. G. Cox, \textit{On the blocks of the
infinitesimal Schur algebras}, Quart. J. Math {\bf 51} (2000),
39-56.

\bibitem{DD}
R. Dipper and S. Donkin, \textit{Quantum $GL_n$}, Proc. London
Math. Soc. {\bf 63} (1991), no. 1, 165--211.

\bibitem{DJ1}
R. Dipper and G. James, \textit{The $q$-Schur algebra},
Proc.London Math.Soc. {\bf 59} (1989), 23-50.

\bibitem{DJ2}
R. Dipper and G. James,\ \textit{$q$-Tensor spaces and q-Weyl
modules}, Trans. Amer. Math. Soc. {\bf 327} (1991), 251-282.

\bibitem{DG2}
S. Doty and A. Giaquinto, \textit{Presenting Schur algebras,}
International Mathematics Research Notices, {\bf 36} (2002), 1907-1944.

\bibitem{DNP1}
S. R. Doty, D. K. Nakano and K. M. Peters, \textit{On Infinitesimal
Schur algebras,} Proc. London Math. Soc. {\bf 72} (1996), 588-612.

\bibitem{D94}
S. Donkin, \textit{On Schur algebras and related algebras IV: The blocks of the Schur algebras},
J. Algebra, {\bf 168} (1994), 400--429.

\bibitem{Don}
S. Donkin, \textit{The $q$-Schur algebra},
London Mathematical Society Lecture Note Series, {\bf 253}.
Cambridge University Press, Cambridge, 1998.


\bibitem{DK}
Y. A. Drozd and V. V. Kirichenko, \textit{Finite-dimensional
algebras}, Translated from the 1980 Russian original and with an
appendix by Vlastimil Dlab. Springer-Verlag, Berlin, 1994.

\bibitem{Du92}
J. Du, \textit{Kahzdan-Lusztig bases and isomorphism
theorems for $q$-Schur algebras,} Contemp. Math. {\bf 139}
(1992), 121-140.

\bibitem{Du}
J. Du, \textit{ A note on the quantized Weyl reciprocity at roots
of unity,} Algebra Colloq. {\bf 2} (1995), 363-372.

\bibitem{Du95}
J. Du, \textit{$q$-Schur algebras, asymptotic forms and quantum $SL_n$,} J. Algebra. {\bf 177} (1995), 385-408.

\bibitem{Du96}
J. Du, \textit{Cells in certain sets of matrices}, T\"{o}hoku Math. J. {\bf 48} (1996), 417-427.


\bibitem{DFW}
J. Du, Q. Fu and J.-p. Wang, \textit{ Infinitesimal quantum
$\frak{gl}_n$ and  little $q$-Schur algebras,} J. Algebra {\bf
287} (2005), 199-233.

\bibitem{DP}
J.Du and B. Parshall, \textit{ Monomial Bases For $q$-Schur
Algebras,} Trans. Amer. Math. Soc. {\bf 355} (2003), no. 4, 1593-1620.


\bibitem{DPW}
J. Du, B. Parshall, and J.-p. Wang,
\textit{Two-parameter quantum linear groups and the hyperbolic
invariance of $q$-Schur algebras,} J.London Math. Soc. {\bf
44}(1991), 420-436.

\bibitem{Erd90}
K. Erdmann, \textit{Blocks of tame representation type and related algebras}, Lecture Notes in Math. {\bf 1428}, Springer-Verlag, New York 1990.

\bibitem {EN}
K. Erdmann and D. K. Nakano \textit{Representation type of
$q$-Schur algebras}, Trans. Amer. Math. Soc. {\bf 353} (2001),
4729-4756.

\bibitem{Fu05}
Q. Fu, \textit{A comparison of infinitesimal and little $q$-Schur
algebras},  Comm. Algebra, {\bf 33} (2005), 2663-2682.

\bibitem{Fu07}
Q. Fu, \textit{Little $q$-Schur algebras at
even roots of unity}, J. Algebra {\bf 311} (2007), 202-215.

\bibitem{Fu08}
Q. Fu, \textit{Finite representation type of infinitesimal $q$-Schur algebras}, Pacific J. Math {\bf 237} (2008), 57-76.

\bibitem{QF2}
Q. Fu, \textit{ Semisimple infinitesimal $q$-Schur
algebras}, Arch. Math., {\bf 90} (2008), 295-303.

\bibitem{GSTF1} A. M. Gainutdinov, A. M. Semikhatov, I. Yu. Tipunin, and B. L. Feigin, \textit{Kazhdan--Lusztig
correspondence for the rpresentation category of the triplet $W$-algebra in logarithmic DFT}. Theoretical and Mathematical Physics {\bf 148} (2006), 1210--1235.

\bibitem{Gr80}
J.~A. Green,
{\em Polynomial Representations of {\rm GL}$_n$}, 2nd ed., with an
appendix on Schensted correspondence and Littelmann paths by K.
Erdmann, J.~A. Green and M. Schocker, Lecture Notes in
Mathematics, no.~830, Springer-Verlag, Berlin, 2007.

\bibitem{Groj} I. Grojnowski, {\em The coproduct for quantum $GL_n$}, preprint, 1992.

\bibitem{KS} H. Kondo and Y. Saito, \textit{Indecomposable decomposition of tensor products of modules over the restricted quantum universal enveloping algebra associated to $\mathfrak{sl}_2$},
    arXive: 0901.4221v3.

\bibitem{Jan87}
J. C. Jantzen, Representations of Algebraic Groups,
Academic Press, Boston (1987).

\bibitem {Jan96}
J. C. Jantzen, \textit{Lectures on quantum groups},
Graduate Studies in Mathematics, {\bf 6}. American Mathematical
Society, Providence, RI, 1996.

\bibitem {JI}
M. Jimbo, \textit{A $q$-analogue of $U(\frak {gl}(N+1))$, Hecke
algebras, and the Yang-Baxter equation,} Lett. Math. Phy. {\bf
11} (1986), 247-252.

\bibitem{L1}
G. Lusztig, \textit{Modular representations and quantum groups,}
Comtemp. Math. {\bf 82} (1989) 59-77.

\bibitem{L2}
G. Lusztig, \textit{Finite dimensional Hopf algebras arising from
quantized universal enveloping algebras,} J. Amer. Math. Soc. {\bf
3}(1990), 257-296.

\bibitem{L00} G. Lusztig, {\em Transfer maps for quantum affine ${\frak {sl}}_n$},
in: Representations and quantizations (Shanghai, 1998), China Higher
Educ. Press, Beijing, 2000, 341--356.

\bibitem{Pi}
R. S. Pierce, \textit{Associative algebras}, Springer-Verlag, New
York 1982.


\bibitem{T}
 M. Takeuchi, \textit{ Some topics on
$Gl_q$(n),} J. Algebra {\bf 147} (1992), 379-410.

\bibitem{T1}
M.Takeuchi, \textit{A two-parameter
quantization of GL(n)(summary),} Proc. Japan Acad. {\bf 66 }
(1990), 112-114.

\bibitem{Th}
L. Thams, \textit{The subcomodule structure of the quantum
symmetric powers},  Bull. Austral. Math. Soc.  {\bf 50}  (1994),
29-39.

\end{thebibliography}
\end{document}